\documentclass[11pt]{article}
\usepackage{amsmath, amsfonts, amsthm, amscd, amssymb, commath}
\usepackage{indentfirst, bm}
\usepackage{geometry}
\usepackage{tikz}
\usepackage[colorlinks=true]{hyperref}
\usepackage{authblk,fancyhdr}

\usetikzlibrary{arrows, positioning, shapes, decorations.markings,decorations.pathmorphing, bending}
\tikzstyle{arrowstyle}=[scale=0.8]
\tikzstyle{directed}=[postaction={decorate,decoration={markings,    mark=at position 0.1 with {\arrow[arrowstyle]{stealth}}}}]
\tikzstyle{midway_mod}=[postaction={decorate,decoration={markings,    mark=at position 0.5 with {\arrow[arrowstyle]{stealth}}}}]

\newtheorem{theorem}{Theorem}[section]
\newtheorem{proposition}[theorem]{Proposition}
\newtheorem{lemma}[theorem]{Lemma}
\newtheorem{corollary}[theorem]{Corollary}

\theoremstyle{definition}
\newtheorem{definition}[theorem]{Definition}

\theoremstyle{remark}
\newtheorem*{remark}{Remark}
\numberwithin{equation}{section}

\title{Some Dynamical Properties on Manifolds with no Conjugate Points}

\author[1]{Fei Liu\thanks{liufei@math.pku.edu.cn}}
\affil[1]{\small College of Mathematics and Systems Science, Shandong University of Science and Technology, Qingdao, 266590, China}

\author[2]{Xiaokai Liu\thanks{liuxk@mail.sustech.edu.cn}}
\affil[2]{\small Department of Mathematics, Southern University of Science and Technology, Shenzhen, 518055, China}

\author[3]{Fang Wang\thanks{fangwang@cnu.edu.cn}}
\affil[3]{\small School of Mathematical Sciences, Capital Normal University, Beijing, 100048, China}

\begin{document}
\maketitle

\begin{abstract}
	In this article, we study the dynamics of geodesic flows on Riemannian (not necessarily compact) manifolds with no conjugate points. We prove the Anosov Closing Lemma, the local product structure, and the transitivity of the geodesic flows on $\Omega_1$ under the conditions of bounded asymptote and uniform visibility. As an application, we further discuss about some generic properties of the set of invariant probability measures.


\end{abstract}

\section{\bf Introduction}

In this article, we always assume that the manifold $(M,g)$ is a (compact or non-compact) manifold with no conjugate points. We want to study some important dynamical properties including the Anosov Closing Lemma, the local product structure, and the topological transitivity of the geodesic flows on the non-wandering set. Under some mild conditions, we prove these properties hold, then show some applications. The concept of manifolds with no conjugate points is a natural extension of the negatively/non-positively curved manifolds, on which the dynamics of geodesic flows have already been extensively studied. The geodesic flows on compact negatively curved manifolds are uniformly hyperbolic flows (cf.~\cite{Ano}), which have very rich dynamics. We suggest~\cite{Kn2, Pa} as a comprehensive discussion on the dynamical properties of these geodesic flows. If we loosen the requirement on the curvature and allow the existence of zero curvature on the manifold, things get much more complicated. The geodesic flow is not uniformly hyperbolic and many tools no longer work. A landmark work on the dynamics of geodesic flows on rank $1$ manifolds with non-positive curvature is Knieper's work~\cite{Kn1}. In that article, Knieper developed a series of new tools and proved the existence and uniqueness of the measure of maximal entropy. From then on, mathematicians have done many ingenious works and presented a lot of deep results on the geodesic flows on non-positively curved manifolds.

It is a natural question to ask if we further loosen the restriction on the curvature to allow some regions admitting positive curvature, what conclusion we can get on the dynamics of the geodesic flows?

Manifolds with no focal points is a natural generalization of this kind. Easy to see that manifolds of negative/non-positive curvature contains no focal points. Although the manifold with no focal points admits the existence of positive curvature, properties are not far away from those on non-positively curved manifolds. In recent years, many results on the dynamics of geodesic flows on non-positively curved manifolds were proved to be still valid on manifolds without focal points (cf.~\cite{CKP2, LLW, LWW, LZ}). However, if further generalize to the manifolds with no conjugate points, things are much different. Features of geodesic flows on such manifolds are much harder to grasp, thus extra conditions are usually needed. Very recent remarkable results on this topic can be found in~\cite{CKW2, CKW}.

In this article we study the geodesic flows on manifolds with no conjugate points. Our goal is to prove the following important dynamical properties: the Anosov Closing Lemma, the local product structure and the topological transitivity. These properties act as fundamental tools in further studying dynamics of geodesic flows.

The study of above properties on rank $1$ manifolds can be traced back to Ballmann~\cite{B}. Related works such as on rank 1 manifolds with no focal points (cf.~\cite{CKP1, GR, LWW, LZ}), on compact manifolds that admit a Riemannian metric of negative curvature (cf.~\cite{CKW}), and on non-compact rank 1 manifolds with non-positive curvature (cf.~\cite{CS1, CS2}), were gradually proved in recent years.

In our work, we study these dynamical properties on manifolds with no conjugate points. Comparing to the case of no focal points, lack of several key geometrical properties (thus (non-uniform) hyperbolicity) is the substantial difficulty. To overcome these obstacles, we achieve several geometrical properties which build up tools in discussing dynamics under the setting of bounded asymptote and uniform visibility, and successfully prove the Anosov Closing Lemma, local product structure and topological transitivity.

As a direct application, we can exhibit some generic results on the invariant probability measures as an extension of the classical work by Coud{\`e}ne-Schapira in~\cite{CS1, CS2} on non-positively curved manifolds. We believe further results of dynamics based on our work await for exploring.

This article is organized in the following way: in Section~\ref{sec2}, we introduce conceptions and notations we use. In Section~\ref{sec3}, we consider the geometry of manifolds with no conjugate points and prove some important results which will play a key role in our work. In Section~\ref{sec4} and~\ref{sec5}, we prove the local product structure, the Anosov Closing Lemma and the topological transitivity. At last in Section~\ref{sec6}, we present some generic properties of the set of invariant probability measures on $\Omega_1/\Omega_{\text{NF}}$ based on the results in previous sections. 

\section{\bf Notations and Preliminary}\label{sec2}
Let $(M,g)$ denote a connected complete $n$-dimensional Riemannian manifold with no conjugate points, and $(\tilde{M},\tilde{g})$ denote its universal cover.

Given a unit vector $v\in T^1M$ with its base point $\pi(v)=p\in M$. Let $\gamma_v$ denote the unit speed geodesic determined by $v$.  Similarly let $\gamma_{x,y}$ denote the connecting geodesic of two points $x,y\in M$ with parametrization $\gamma_{x,y}(0)=x$. $\phi_t(v)=\gamma'_v(t)$ denote the parallel transport of $v$ along $\gamma_v$.

Denote $d(\cdot,\cdot)$ the distance induced by the Riemannian metric. Define $d_1$ the Knieper Metric on $T^1\tilde{M}$.

\begin{definition}[Knieper Metric]
	\begin{displaymath}
		d_1(v,w)=\max_{0\leq t\leq 1} d(\gamma_v(t),\gamma_w(t)),\quad\forall v,w\in T^1\tilde{M}.
	\end{displaymath}
\end{definition}

Given two different geodesics $\gamma_1, \gamma_2$ on $\tilde{M}$, we call them positively asymptotic if there exist a positive number $C$ such that
\begin{displaymath}
	d(\gamma_1(t),\gamma_2(t))\leq C,\quad\forall t\geq 0
\end{displaymath}

Similarly, we say they are negatively asymptotic if the above inequality holds for all $t\leq 0$, and bi-asymptotic if they are both positively and negatively asymptotic.

Easy to check that positively/negatively asymptote builds an equivalent relation on geodesic on $\partial \tilde{M}$. We denote this positively/negatively asymptotic class by $\gamma(+\infty)$ and $\gamma(\text{--}\infty)$. Let $\partial\tilde{M}$ denote the ideal boundary of the universal cover, which is the collection of all equivalent classes on $\partial \tilde{M}$.

Define the angle between two geodesics by $\measuredangle_p(\gamma_{p,x},\gamma_{p,y})=\measuredangle(\gamma'_{p,x}(0),\gamma'_{p,y}(0))$. We define the cone as following:

\begin{itemize}
	\item $C(v,\epsilon)=\{a \in \tilde{M}\cup\partial{\tilde{M}}-\{p\}\mid \measuredangle_p(\gamma_v,\gamma_{p,a})< \epsilon\}$.
	\item $C_{\epsilon}(v) =C(v,\epsilon)\cap \partial{\tilde{M}}= \{\gamma_{w}(+\infty)\mid w \in  T^{1}\tilde{M}, \measuredangle(v,w)<\epsilon\}$.
	\item $TC(v,\epsilon,r)= \{q \in \tilde{M}\cup\partial{\tilde{M}}\mid \measuredangle_{p}(\gamma_{v}(+\infty),\gamma_{p,q})< \epsilon, d(p,q)>r\}$.
\end{itemize}

$TC(v,\epsilon,r)$ is called the \emph{truncated cone} with axis $v$ and angle $\epsilon$. Obviously $\gamma_{v}(+\infty) \in TC(v,\epsilon,r)$. There is a unique topology $\tau$ on $\tilde{M}\cup\partial{\tilde{M}}$ such that for each $\xi \in \partial{\tilde{M}}$ the set of truncated cones containing $\xi$ forms a local basis for $\tau$ at $\xi$. This topology is usually called the \emph{cone topology}. Under this topology, $\tilde{M}\cup\partial{\tilde{M}}$ is homeomorphic to the closed unit ball in $\mathbb{R}^{\text{dim}(M)}$, and the ideal boundary is homeomorphic to the unit sphere $\mathbb{S}^{\text{dim}(\tilde{M})-1}$. For more details about the cone topology, see~\cite{Eb1}.

Next, we define the stable (unstable) manifolds of a geodesic flow both on $T^1 M$ and $T^1\tilde{M}$.

\begin{definition}[Stable Manifolds]
	Give $v\in T^1M$, define the stable manifolds and local stable manifolds by
	\begin{displaymath}
		\begin{aligned}
		W^{ss}(v)&=\{w\in T^1M\mid \lim_{t\to\infty} d(\gamma_v(t),\gamma_w(t))=0\},\\
		W^{ss}_{\epsilon}(v)&=\{w\in W^{ss}(v)\mid d(\gamma_v(t),\gamma_w(t))\leq\epsilon, \forall t\geq0 \}.\\
		\end{aligned}
	\end{displaymath}

	By replacing $\gamma_v(t)$ with $\gamma_v(-t)$, we can define the unstable manifolds $W^{su}(v)$ and $W^{su}_{\epsilon}(v)$.

	Similarly define the stable (unstable) manifolds on the universal cover $\tilde{M}$ by lifting the unit vector $v$ to $\tilde{v}\in T^1\tilde{M}$. We will use the same notation $W^{ss}(\tilde{v})$ and $W^{su}(\tilde{v})$ for convenient.
\end{definition}

\begin{definition}[Busemann Function]
	Given $v\in T^1\tilde{M}$, we define the Busemann function
	\begin{displaymath}
		\begin{aligned}
			f_v:\tilde{M}&\to\mathbb{R},\\
			p&\mapsto f_v(p)=\lim_{t\to\infty}(d(p,\gamma_v(t))-t).
		\end{aligned}
	\end{displaymath}
\end{definition}

\begin{definition}[Stable Horocycle]
	For $\tilde{v}\in T^1\tilde{M}$, define the stable horocycle by
	\begin{displaymath}
		W^{SS}(\tilde{v})=\{\tilde{w}\in T^1_q\tilde{M}\mid f_{v}(q)=0, \tilde{w}=-\nabla f_{\tilde{v}}(q)\}.
	\end{displaymath}

	And the unstable horocycle $W^{SU}$ in a similar way.
\end{definition}

\begin{remark}
	Here the stable/unstable horocycle actually consists the unit inner/outer normal vector on the horocycle. It is defined on the unit tangent bundle. We use $W^{SS}/W^{SU}$ to denote the stable/unstable horocycle while $W^{ss}/W^{su}$ to denote the stable/unstable manifolds. In general, they might not be the same.
\end{remark}

Let $J^\bot(v)$ denote the $2(n-1)$-dimensional vector space of normal Jacobi fields along the geodesic $\gamma_v$. For each $w\in T_p M$ perpendicular to $v$ and fixed time parameter $t\neq 0$, denote by $J_{w,t}$ the unique Jacobi field satisfied that

\begin{displaymath}
	J_{w,t}(0)=w,\qquad J_{w,t}(t)=0
\end{displaymath}

Such Jacobi field exits and $J_{w,t}\in J^{\bot}(v), \forall t\neq 0$. In~\cite{G} Green proved that the two limiting vector fields
\begin{displaymath}
	\lim_{t\to\infty}J_{w,t}, \lim_{t\to\text{--}\infty}J_{w,t}.
\end{displaymath}
always exist and both are Jacobi fields along the geodesic $\gamma_v$.

\begin{definition}
	$\forall v\in T^1_p M$, denote the stable and unstable Jacobi fields by
	\begin{displaymath}
		\begin{aligned}
			J^s(v)&=\{J\in J^{\bot}(v)\mid\exists w\in T_{p}M, w\bot v,\lim_{t\to\infty}J_{w,t}=J\},\\
			J^u(v)&=\{ J\in J^{\bot}(v)\mid\exists w\in T_{p}M, w\bot v, \lim_{t\to\text{--}\infty} J_{w,t}=J\}.
		\end{aligned}
	\end{displaymath}
\end{definition}

\begin{definition}
	Let $J^c(v)=J^s(v)\cap J^u(v)$. We call $J\in J^c(v)$ the central Jacobi field along $\gamma_v$.
\end{definition}

\begin{definition}[Bounded Asymptote]
	Given a manifold with no conjugate points. We call this manifold satisfies \emph{bounded asymptote}, if there exist a uniform positive constant $C$ for all geodesics $\gamma$ and any stable Jacobi field $J$ along $\gamma$
	\begin{displaymath}
		\norm{J(t)}\leq C\norm{J(0)}, \quad\forall t\geq 0.
	\end{displaymath}
\end{definition}

The concept \emph{bounded asymptote} was first introduced by Eschenburg in~\cite{Es}. Such manifolds are still of great interest and contain rich types of manifolds. It is known that manifolds of non-positive curvature and manifolds without focal points both admit the bounded asymptotic condition. Moreover, we can calculate the constant $C=1$ for these cases, because the distance between positively asymptotic geodesics is non-increasing. In~\cite{LW} Liu-Wang proved that the geodesic flow on manifold of bounded asymptote is entropy expansive.

Next, we want to define the rank on the manifolds with no conjugate point. The definition of rank in the setting of non-positively curved manifolds was introduced by Ballman-Brin-Eberlein in~\cite{BBE}. Due to the flat strip theorem, the definition can be easily migrated to the manifolds with no focal points. Rank of the manifolds became an important concept in both geometry and dynamical systems. Dynamics on rank $1$ manifolds with non-positive curvature or with no focal points share similar properties with the ones on manifolds with negative curvature. Refer to~\cite{BCFJ, CT, CKP1, CKP2, GR, Wu} for recent works.

But in the case of manifolds with no conjugate points, there is no uniform definition of rank. We note that Knieper-Peyerimhoff in~\cite{KP} and Rifford-Ruggiero in~\cite{RR} both give definition of the rank separately, but their definitions are different.

Luckily under the bounded asymptote condition, Eschenburg in~\cite{Es} gives a sufficient condition for the existence of non-trivial central Jacobi field.

\begin{theorem}[Eschenburg]\label{thm_1}
	Let $M$ be a manifold with no conjugate points which satisfies bounded asymptote. Let $w,v\in T^1M$ and $w\neq\phi_t(v)$ for all $t$. If the geodesic $\gamma_w$ is bi-asymptotic to $\gamma_v$, then there exist central Jacobi field $J\in J^c(v)$ along $\gamma_v$.
\end{theorem}

Based on Eschenburg's result, Ruggiero-Meneses in~\cite{RM} gives the following definition of the rank, which is similar to the definition on non-positively curved manifold.

\begin{definition}
	Let $M$ be a manifold of bounded asymptote, $v\in T^1M$. Define the rank of $v$ by
	\begin{displaymath}
		\textrm{rank}(v)=\dim J^c(v)+1
	\end{displaymath}

	And define the rank of the manifold by
	\begin{displaymath}
		\textrm{rank}(M)=\min_{v} \textrm{rank}(v)
	\end{displaymath}
\end{definition}

We will use this form of rank through the article. For more information, refer to~\cite{RM}.

Inspired by the famous flat strip theorem on manifolds of non-positive curvature or of no focal points, we introduce the following definition.

\begin{definition}
	Given a geodesic $\gamma$, we call $\gamma$ bounds a bounded strip if there exists another geodesic $\beta\neq\gamma$ bi-asymptotic to $\gamma$.
\end{definition}

From this definition, we can see that if $\gamma_v$ bounds a bounded strip, $\textrm{rank}(v)>1$. And if $\textrm{rank}(v)=1$, it does not bound one. But when $\gamma_v$ dos not bound a bounded strip, there is no guarantee that $\textrm{rank}(v)=1$.

Denote by $\Omega\subseteq T^1M$ the set of non-wandering vectors. For any $v\in\Omega$ and a neighborhood $V\ni v$, there is a sequence of increasing number $t_n\to\infty$ such that $\phi_{t_n}(V)\cap V\neq\emptyset$.

Let $\Omega_{\text{NF}}\subseteq\Omega$ denote the set of non-wandering vectors which do not bound a bounded strip, and let $\Omega_1\subseteq\Omega$ be the set of rank $1$ non-wandering unit vector, whose lift has the property that its stable (unstable) manifold $W^{ss}(v)$~(or $W^{su}(v)$) coincide with the stable (unstable) horocycle $W^{SS}(v)$~(or $W^{SU}(v)$) in $T^1\tilde{M}$.

Then the following inclusion relation naturally holds by definition
\begin{displaymath}
	\Omega_1\subseteq\Omega_{\text{NF}}\subseteq\Omega.
\end{displaymath}

In addition, we define $\Omega^{\text{rec}}_1\subseteq\Omega$ be the set of all rank $1$ recurrent unit vectors. We will prove that in our setting of manifolds, $\Omega^{\text{rec}}_1\subseteq\Omega_1$ in Proposition~\ref{thm_3_3} in Section~\ref{sec3}.

Another geometric property playing an essential rule of dynamics to our results is (uniform) visibility given in~\cite{Eb1, Eb2}.

\begin{definition}[Visibility]
	Say $\tilde{M}$ satisfies the Visibility Axiom if for any $p\in\tilde{M}$, $\epsilon>0$, there exists $R=R(p,\epsilon)>0$, such that for any geodesic $\gamma:]a,b[\to\tilde{M}$ with $d(p,\gamma)\geq R$, we have $\measuredangle_p(\gamma(a),\gamma(b))\leq\epsilon$. Here we allow $a$ and $b$ to be infinity.

	Say $\tilde{M}$ satisfies the Uniform Visibility Axiom if the constant $R$ can be chosen independently of $p$. Say $M$ satisfies (Uniform) Visibility Axiom if its universal cover $\tilde{M}$ does.
\end{definition}

An intuitive way to understand visibility is to observe geodesics at a given point $p\in\tilde{M}$. The further the geodesic is, the narrower the angle of view is.

For manifolds satisfying Visibility Axiom, we have the following property:

\begin{theorem}[Eberlein-O' Neil~\cite{Eb2, Eb3}]\label{thm_2_1}
	Given $\tilde{M}$ a manifold satisfying Visibility Axiom and any two points $x\neq y\in\partial\tilde{M}$, there is a connecting geodesic $\gamma_{x,y}$.
\end{theorem}

Such property acts as a strong tool in studying dynamics of geodesics. A natural question is raised that what kind of manifolds admit Visibility Axiom? Luckily, we know that a rich bunch of surfaces have this property.

\begin{theorem}[Eberlein~\cite{Eb1}]
	A closed surface with no conjugate points and genus no less than $2$ satisfies the Uniform Visibility Axiom.\label{thm_2_2}
\end{theorem}

\section{\bf Geometric Properties of Rank 1 Manifolds of Bounded Asymptote}\label{sec3}

As a preparation to the following discussion on dynamics, we will prove some geometric results on $\tilde{M}$ in this part. The main results in this chapter were proven in the case that manifolds with no conjugate points with an extra condition bounded asymptote. Similar conclusion on non-positively curved manifolds were proven in~\cite{B}. For the case rank $1$ manifolds with no focal points, refer to~\cite{LWW, Wa}.

We start with some helpful conclusions proved by Ruggiero.

\begin{lemma}[Ruggiero]\label{lem_3_0}
	Let $M$ be a compact manifold with no conjugate points and bounded asymptote. The following results hold:

	1. (\cite{Ru1}, Lemma 1.1) There exists $A>0$ such that for any geodesics $\gamma_1$ and $\gamma_2$ in $\tilde{M}$ with $\gamma_1(+\infty)=\gamma_2(+\infty)$, we have that
	\begin{displaymath}
		d(\gamma_1(t),\gamma_2(t))\leq A\cdot d(\gamma_1(0),\gamma_2(0)),\quad\forall t\geq0.
	\end{displaymath}

	2. (\cite{Ru2}, Lemma 1.5) There exists $C>0$ such that for any $p\in\tilde{M}$ and $v,w\in T^{1}_p\tilde{M}$, we have that
	\begin{displaymath}
		d(\gamma_v(t),\gamma_w(t))\geq C\measuredangle_p(v,w)\cdot t,\quad\forall t\geq0.
	\end{displaymath}

	3. (\cite{RM}, P87) For any $p\in\tilde{M}$ and $\xi\in\partial\tilde{M}$, there exist a unique geodesic connecting $p$ and $\xi$.
\end{lemma}

\begin{theorem}[Continuity at Infinity]\label{thm_3_1}
	Let $\tilde{M}$ be a simply connected manifold with no conjugate points. Suppose $\tilde{M}$ satisfies bounded asymptote. For a sequence of unit vectors $v_n\to v\in T^1\tilde{M}$ and a sequence of increasing numbers $t_n\to+\infty$, we have

	\begin{displaymath}
		\lim_{n\to\infty} \gamma_{v_n}(t_n)=\gamma_{v}(+\infty).
	\end{displaymath}
\end{theorem}

\begin{proof}
	Let $x_n=\pi(v_n)$, $x=\pi(v)$ and $y_n=\gamma_{v_n}(t_n)$. Let $s_n=d(x,y_n)$. Denote $w_n={\gamma'}_{x,y_n}(0)$ and $u_n={\gamma'}_{x,\gamma_{v_n}(+\infty)}(0)$. Let $\alpha_n=\measuredangle_x(u_n,w_n)$. Here the existence of $u_n$ is guaranteed by Lemma~\ref{lem_3_0}.

	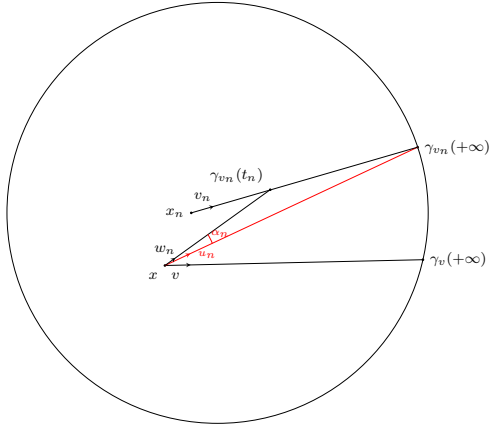
\begin{figure}[htbp]
		\centering
		\scalebox{0.7}{
			\begin{tikzpicture}
				\draw (0,0) circle (4);
				\coordinate [label=below left:$\scriptstyle x$] (X) at (-1,-1);
				\node at (X) [circle, fill, inner sep=0.5pt]{};
				\coordinate [label=right:$\scriptstyle \gamma_v(+\infty)$] (A) at (3.9,-0.89);
				\node at (A) [circle, fill, inner sep=0.5pt]{};
				\coordinate [label=right:$\scriptstyle \gamma_{v_n}(+\infty)$] (B) at (3.8,1.249);
				\node at (B) [circle, fill, inner sep=0.5pt]{};
				\draw[directed] (X)--(A);
				\draw[red, directed] (X)--(B);
				\coordinate [label=left:$\scriptstyle x_n$] (Y) at (-0.5,0);
				\node at (Y) [circle, fill, inner sep=0.5pt]{};
				\draw[directed] (Y)--(B);
				\coordinate [label=above left:$\scriptstyle \gamma_{v_n}(t_n)$] (Z) at (1,0.44);
				\node at (Z) [circle, fill, inner sep=0.5pt]{};
				\draw[directed] (X)--(Z);
				\node at (-0.28,0.3) {$\scriptstyle v_n$};
				\node at (-1,-0.7) {$\scriptstyle w_n$};
				\node at (-0.8,-1.2) {$\scriptstyle v$};
				\node at (-0.2,-0.8) [red] {$\scriptscriptstyle u_n$};
				\draw[red] (-0.1,-0.58) arc (20:34:0.8);
				\node[red] at (0.05,-0.4) {$\scriptscriptstyle \alpha_{n}$};
			\end{tikzpicture}
		}
		\caption{Continuity at Infinity}\label{fig1}
	\end{figure}

	Using the triangular inequality, we have $|s_n-t_n|\leq d(x,x_n)\leq d_1(v,v_n)\to0$.

	Lemma~\ref{lem_3_0} allows for a constant $A>0$ such that
	\begin{displaymath}
		d(\gamma_{u_n}(t_n),\gamma_{v_n}(t_n))\leq A\cdot d(x,x_n)\to 0.
	\end{displaymath}

	Therefore, we have
	\begin{equation}\label{eq_3_0_1}
		\begin{aligned}
			d(\gamma_{w_n}(s_n),\gamma_{u_n}(t_n))&=d(\gamma_{v_n}(t_n),\gamma_{u_n}(s_n))\\
												  &\leq d(\gamma_{v_n}(t_n),\gamma_{u_n}(t_n))+|t_n-s_n|\\
												  &\leq (A+1) d(x,x_n).
		\end{aligned}
	\end{equation}

	Again, Lemma~\ref{lem_3_0} guarantee a constant $C>0$ such that
	\begin{equation}\label{eq_3_0_2}
		C\cdot\alpha_n\cdot s_n\leq d(\gamma_{u_n}(s_n),\gamma_{w_n}(s_n)).
	\end{equation}

	Together with Equation~\eqref{eq_3_0_1} and~\eqref{eq_3_0_2}, we can conclude the angle is controlled by
	\begin{displaymath}
		\alpha_n\leq\frac{A+1}{Cs_n}d(x,x_n).
	\end{displaymath}

	As $s_n\to\infty$ but $d(x,x_n)\to0$, the angle $\alpha_n\to0$.

	On the other hand, we know that
	\begin{equation}
		d_1(u_n,v_n)=\max_{0\leq t\leq 1}d(\gamma_{u_n}(t),\gamma_{v_n}(t))\leq Ad(x,x_n)\to0.
	\end{equation}

	Thus $v_n\to v$ implies that $u_n\to v$. We have $\measuredangle_x(u_n,v)\to 0$. Therefore $\measuredangle_x(w_n,v)=\measuredangle_x(w_n,u_n)+\measuredangle_x(u_n,v)\to0$, which suggests $w_n\to v$. Choose $n$ sufficiently large, we have that $\gamma_{w_n}(s_n)$ is contained in the truncated cone:
	\begin{displaymath}
		\gamma_{w_n}(s_n)\in TC(v,\frac{s_n}{2},2\measuredangle_x(w_n,v)).
	\end{displaymath}

	Let $s_n\to\infty$, $\measuredangle_x(v,w_n)\to0$, we have that
	\begin{displaymath}
		\gamma_{v_n}(t_n)=\gamma_{w_n}(s_n)\to\gamma_v(+\infty).
	\end{displaymath}
\end{proof}

\begin{proposition}\label{thm_3_3}
	Let $\tilde{M}$ be a simply connected manifold with no conjugate points. Suppose $\tilde{M}$ satisfies bounded asymptote. Let $v\in T^1\tilde{M}$ be a rank $1$ recurrent vector, we have $W^{ss}(v)=W^{SS}(v)$ coincided. That is to say, for any $w\in W^{SS}(v)$
	\begin{displaymath}
		\lim_{t\to+\infty} d_1(\phi_t(v),\phi_t(w))=0.
	\end{displaymath}
\end{proposition}

\begin{proof}
	$v$ is recurrent, there exists a sequence of isometry $\{g_n{\}}_{n=1}^\infty\subseteq\Gamma$ and s sequence of increasing number ${\{t_{n}\}}_{n=1}^{\infty}$ with $t_{n}\to\infty$ such that
	\begin{displaymath}
		\mathrm{d}g_n\phi_{t_n}(v)\to v,\quad n\to+\infty.
	\end{displaymath}

	In~\cite{Es} Eschenburg showed that if $w\in W^{SS}(v)$, the corresponding geodesics $\gamma_w$ and $\gamma_v$ are asymptotic on manifolds with no conjugate points and of bounded asymptote.

	Therefore, we can choose a constant $D>0$ such that $d(\gamma_w(t),\gamma_v(t))\leq D,\quad t\geq 0$. Thus $d_1(\phi_t(w),\phi_t(v))\leq D,\quad t\geq 0$.

	We want to prove the limit converges to $0$ by contradiction.
	\begin{equation}\label{eq_3_1}
		\lim_{t\to+\infty}d_1(\phi_t(w),\phi_t(v))=0.
	\end{equation}

	Assume the~\eqref{eq_3_1} fails. In~\cite{Ru1} Ruggiero proved that $\exists A>0, \forall t>s\geq0$,
	\begin{displaymath}
		d(\gamma_v(t),\gamma_w(t))\leq A\cdot d(\gamma_v(s),\gamma_w(s)).
	\end{displaymath}

	Above inequality implies that for any sequence $t_n\to\infty$, $\lim_{n\to\infty} d(\gamma_w(t_n),\gamma_v(t_n))\neq 0$. Otherwise~\eqref{eq_3_1} will hold. When the limit is not $0$, the distance between geodesics is bounded away from 0, as well as the $d_1(\gamma_v(t),\gamma_w(t))$. We can find a constant $a>0$ such that
	\begin{equation}\label{eq_2}
		d_1(\phi_t(v),\phi_t(w))\geq a,\quad t\geq0.
	\end{equation}

	Thus
	\begin{equation}\label{eq_3_3}
		a\leq d_1(\phi_{t_n+s}(v),\phi_{t_n+s}(w))\leq A\cdot d_1(v,w),\quad\forall s\in [-t_n,\infty).
	\end{equation}

	Since isometry maps geodesics to geodesics, for any isometry $g_n$, we have that $\phi_t(\mathrm{d}g_{n}v)=g_n\phi_t(v)$. Therefore
	\begin{displaymath}
		\begin{aligned}
			d_1(\phi_{t_n+s}(v),\phi_{t_n+s}(w))&=d_1(\mathrm{d}g_n\phi_{t_n+s}(v),\mathrm{d}g_n\phi_{t_n+s}(w))\\
												&=d_1({\frac{\mathrm{d}}{\mathrm{d}t}|}_{t=0}g_n\circ \pi\phi_{t_n+s+t}(v),{\frac{\mathrm{d}}{\mathrm{d}t}|}_{t=0}g_n\circ \pi\phi_{t_n+s+t}(w),)\\
												&=d_1(\phi_s\circ{\frac{\mathrm{d}}{\mathrm{d}t}|}_{t=0}g_n\circ\pi\phi_{t_n+t}(v),\phi_s\circ{\frac{\mathrm{d}}{\mathrm{d}t}|}_{t=0} g_n\circ\pi\phi_{t_n+t}(w))\\
												&=d_1(\phi_s\circ\mathrm{d}g_n\phi_{t_n}(v),\phi_s\circ\mathrm{d}g_n\phi_{t_n}(w)).
		\end{aligned}
	\end{displaymath}

	By~\eqref{eq_3_3}, we get
	\begin{equation}\label{eq_3_4}
		a\leq d_1(\phi_s\circ\mathrm{d}g_n\phi_{t_n}(v),\phi_s\circ\mathrm{d}g_n\phi_{t_n}(w))\leq A\cdot d_1(v,w),\quad\forall s\geq-t_n.
	\end{equation}

	Because $v$ is recurrent, $\mathrm{d}g_n\phi_{t_n}(v)$ return to the neighborhood of $v$. For any $\epsilon>0$, there exists $N>0$, such that $d_1(\mathrm{d}g_n\phi_{t_n}(v),v)<\epsilon$, $\forall n>N$. Then~\eqref{eq_3_4} implies that
	\begin{displaymath}
		a-\epsilon\leq d_1(v,\mathrm{d}g_n\phi_{t_n}(w))\leq A\cdot d_1(v,w)+\epsilon.
	\end{displaymath}

	By passing to a sub-sequence if needed, we can assume that
	\begin{displaymath}
		\mathrm{d}g_n\phi_{t_n}(w)\to w'\in T^1\tilde{M}.
	\end{displaymath}

	Let $n\to+\infty, t_n\to+\infty$, we have
	\begin{displaymath}
		a\leq d_1(\phi_s v,\phi_s w')\leq A\cdot d_1(v,w),\quad s\in\mathbb{R}.
	\end{displaymath}

	Because the distance $d_1(v,\mathrm{d}g_n\phi_{t_n}(w))$ is bounded away from $0$, we know that $w'\neq\gamma'_v(t)$. Therefore $\gamma_{w'}\neq\gamma_v$. But these two geodesics are bi-asymptotic. By Theorem~\ref{thm_1}, $J^c(v)$ is non-trivial, $\textrm{rank}(v)\geq 2$, contradicts to the assumption that $v$ is rank $1$.
\end{proof}

Since recurrent vectors are non-wandering vectors, we know any rank $1$ recurrent vector $v\in T^1M$ must lies in $\Omega_1$, or $\Omega^{\text{rec}}_1\subseteq\Omega_1$.

\begin{theorem}\label{thm_3_2}
	Let $\tilde{M}$ be a simply connected manifold with no conjugate points and of bounded asymptote, $v\in T^1\tilde{M}$ is a rank $1$ unit vector. Then for any $\epsilon>0$, there exist $U_\epsilon, V_\epsilon$ neighborhoods of $\gamma_v(\text{--}\infty)$ and $\gamma_v(+\infty)$ on the boundary respectively, such that for any pair of $(\xi,\eta)\in U_\epsilon\times V_\epsilon$, there exists a unique rank $1$ geodesic $\gamma_{\xi,\eta}$ connecting them. Moreover, we can require this geodesic is close to $\gamma_v$ at time $0$ in the sense $d(\gamma_v(0),\gamma_{\xi,\eta})<\epsilon$.
	\begin{figure}
		\centering
		\scalebox{0.7}{
			\begin{tikzpicture}
				\draw (0,0) circle (1);
				\node at (0,0) [circle, fill, inner sep=1pt]{};
				\node at (0,0.5) {$\scriptstyle \gamma(0)$};
				\node at (1.6,0.5) {$\scriptstyle B(\gamma(0),\epsilon)$};
				\coordinate [red,label=right:$\scriptstyle V_\epsilon$] (E) at (3,0);
				\coordinate [red,label=left:$\scriptstyle U_\epsilon$] (S) at (-3,0);
				\draw (S) -- (E);
				\draw[red] (2.955,-0.525) arc (-10:10:3);
				\draw[red] (-2.955,0.525) arc (170:190:3);
				\coordinate [label=below right:$\scriptstyle\eta$] (E) at (2.99,-0.35);
				\node at (E) [circle, fill, inner sep=0.5pt] {};
				\coordinate [label=below left:$\scriptstyle\xi$] (X) at (-2.99,-0.35);
				\node at (X) [circle, fill, inner sep=0.5pt] {};
				\draw[directed, blue] (X).. controls (-1,-1) and (0,-0.5).. (E);
			\end{tikzpicture}
		}	
		\caption{Theorem~\ref{thm_3_2}}\label{fig0}
	\end{figure}
\end{theorem}

The following lemma is needed for the proof.

\begin{lemma}\label{lem_3_2}
	Let $\tilde{M}$ as stated in Theorem~\ref{thm_3_2}. $v\in T^1\tilde{M}$. If there exist a constant $b>0$ such that~$\forall k\in\mathbb{Z}^+$, we can find a pair of points $(p_k,q_k)\in C(-v,\frac{1}{k})\times C(v,\frac{1}{k})$ with $d(\gamma_v(0),\gamma_{p_k,q_k})\geq b$. Then $\textrm{rank}(v)\geq 2$.
\end{lemma}

\begin{figure}[htbp]
	\centering
	\scalebox{0.7}{
		\begin{tikzpicture}
			\draw (0,0) circle (4);
			\draw[blue] (0,0) circle (0.5);
			\node at (1.2,0) {$\scriptscriptstyle B(\gamma_{v}(0),b)$};
			\coordinate [label = left:$v$] (V) at (0,0);
			\coordinate [label = above:$\gamma_v(+\infty)$] (A) at (0,4);
			\coordinate [label = below:$\gamma_v(\text{--}\infty)$] (B) at (0,-4);
			\node at (V) [circle,fill,inner sep=2pt]{};
			\draw[directed] (V) -- (A);
			\draw (B) -- (V);
			\draw[green] (V) -- (0.8,3.92) node [text=black, midway, right] {$C(v,\frac{1}{k})$};
			\draw[green] (V) -- (-0.8,3.92);

			\draw[red] (V) -- (0.8,-3.92) node [text=black, midway, right] {$C(-v,\frac{1}{k})$};
			\draw[red] (V) -- (-0.8,-3.92);
			\coordinate [label=below:$p_k$] (P) at (-0.3,-2.7);
			\node at (P) [circle, fill, inner sep=1pt] {};
			\coordinate [label=above:$q_k$] (Q) at (-0.3,3);
			\node at (Q) [circle, fill, inner sep=1pt] {};
			\draw[blue, rounded corners=0.3cm] (P)--(-0.7,-0.5)--(-1,0)--(-0.5,1.5)--(Q);
		\end{tikzpicture}
	}
	\caption{Both Sequence of Points Bounded}\label{fig2}
\end{figure}
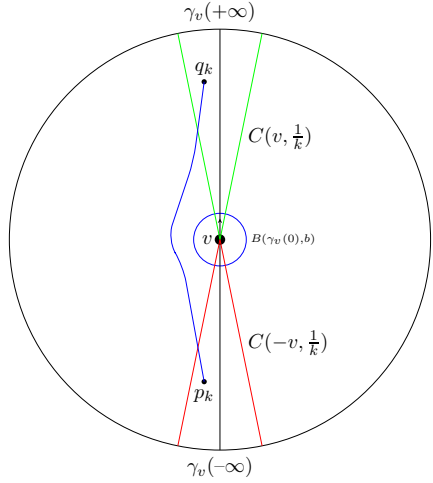

\begin{proof}
	First, we want to prove that at least one of the following two equations holds.
	\begin{equation}
		\lim_{k\to\infty} d(\gamma_v(0),p_k)=\infty,\quad \lim_{k\to\infty} d(\gamma_v(0),q_k)=\infty.
	\end{equation}

	Assume the contrary, we can find some constant $B>0$ such that
	\begin{equation}\label{eq_3_5}
		\left\{
			\begin{aligned}
				d(\gamma_v(0),p_k)&\leq B\\
				d(\gamma_v(0),q_k)&\leq B
			\end{aligned}
		\right. \quad k=1,2,\cdots
	\end{equation}

	As $k$ increases, the angle between $\gamma_v$ and $\gamma_{\pi(v),p_k}$ decreases to zero. Therefore, by choosing a proper sub-sequence, we can find some point $\gamma_v(-t_1)$ on the geodesic $\gamma_v|_{[-B,0]}$ such that $\lim_{i\to\infty} p_{k_i}=\gamma_v(-t_1)$.

	Similarly, by choosing a proper sub-sequence, we can find another point $\gamma_v(t_2)$ as the limit point of $q_{k_j}$. For neater notation, we write $p_k$ and $q_k$ instead of $p_{k_i}$ and $q_{k_j}$.

	Parametrize the geodesic segment $\gamma_{p_k,q_k}$ by setting $d(\pi(v), \gamma_{p_k,q_k}(0))=d(\pi(v), \gamma_{p_k,q_k})\geq b>0$ by assumption, and Equation~\eqref{eq_3_5} shows that $d(p,\gamma_{p_k,q_k})$ has a uniform upper bound $B$ for all $k$.

	Thus, the unit tangent vector $\gamma'_{p_k,q_k}(0)$ must have a converging sub-sequence, denote the limit by $v'\neq v$. We have $\gamma_v(-t_1)$ and $\gamma_v(t_2)$ are on this limiting geodesic $\gamma_{v'}$, contradicting to the assumption that $\tilde{M}$ has no conjugate points.
	\begin{figure}[htbp]
		\centering
		\scalebox{0.7}{
			\begin{tikzpicture}
				\draw (0,0) circle (4);
				\draw[blue] (0,0) circle (0.45);
				\node at (1.2,0) {$\scriptscriptstyle B(\gamma_{v}(0),b)$};
				\coordinate [label = left:$v$] (V) at (0,0);
				\coordinate [label = above:$\gamma_v(+\infty)$] (A) at (0,4);
				\coordinate [label = below:$\gamma_v(\text{--}\infty)$] (B) at (0,-4);
				\node at (V) [circle,fill,inner sep=1pt]{};
				\draw[directed] (V) -- (A);
				\draw (B) -- (V);
				\draw[green] (V) -- (0.8,3.92) node [text=black, midway, right] {$C(v,\epsilon)$};
				\draw[green] (V) -- (-0.8,3.92);

				\draw[cyan] (V) -- (0.8,-3.92) node [text=black, midway, right] {$C(-v,\epsilon)$};
				\draw[cyan] (V) -- (-0.8,-3.92);
				\coordinate [label=below:$p_k$] (P) at (-0.3,-2.7);
				\node at (P) [circle, fill, inner sep=0.5pt] {};
				\coordinate [label=above:$q_k$] (Q) at (-0.3,2.8);
				\node at (Q) [circle, fill, inner sep=0.5pt] {};
				\draw[blue, rounded corners=0.3cm] (P)--(-0.7,-0.5)--(-1,0)--(-0.5,1.5)--(Q);
				\coordinate [label=right:$\scriptstyle{\gamma_{v}(-t_1)}$] (S) at (0,-3);
				\node at (S) [circle, fill, inner sep=0.5pt] {};
				\coordinate [label=right:$\scriptstyle{\gamma_{v}(t_2)}$] (E) at (0,3);
				\node at (E) [circle, fill, inner sep=0.5pt] {};
				\coordinate (W) at (-0.54,0);
				\node at (-0.63,-0.1) {$\scriptscriptstyle{v'}$};
				\draw[red, rounded corners=0.2cm] (S)--(-0.3,-0.9)--(W)--(-0.2,1.3)--(E);
				\draw[red, ->] (-0.51,0) -- (-0.5,0.4);
			\end{tikzpicture}
		}
		\caption{Only One Sequence of Points Bounded}\label{fig3}
	\end{figure}
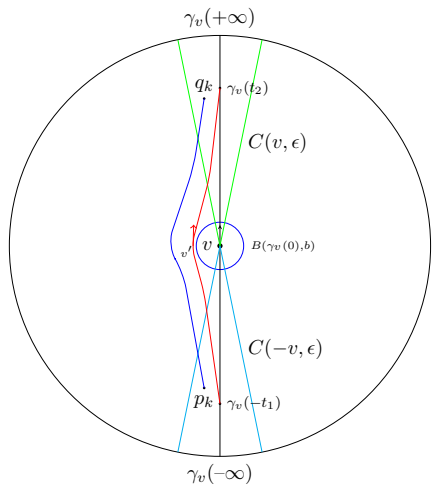

	Next, we want to rule out the case that only one of the distances between $\gamma_v(0)$ and $p_k$ or $q_k$ be bounded above. Without loss of generality, we will show the following statement cannot hold.
	\begin{equation}\label{eq_3_6}
		\exists B>0\quad\text{such that}\quad \left\{
			\begin{aligned}
				d(\gamma_v(0),p_k)&\leq B\\
				d(\gamma_v(0),q_k)&\to\infty
			\end{aligned}
		\right.
	\end{equation}

	Passing to a sub-sequence if needed. Similar to the proof above, we can find the limiting point of $p_k\to\gamma_v(-t_1)$. And $q_k$, we have $q_k\in C(v,\frac{1}{k})$, which implies that $q_k\to\gamma_v(+\infty)$. Therefore, we construct two geodesic segments starting from $\gamma_v(-t_1)$ to $\gamma_v(+\infty)$, contradicting to the divergent property on manifolds with bounded asymptote~\cite{Ru1}.

	Therefore, we can conclude that
	\begin{equation}\label{eq_3_7}
		\left\{
			\begin{aligned}
				\lim_{k\to\infty} d(\gamma_v(0),p_k)&=\infty\\
				\lim_{k\to\infty} d(\gamma_v(0),q_k)&=\infty
			\end{aligned}
		\right.
	\end{equation}

	Now choose $\bar{q}_k$ and $\bar{p}_k$ on $\gamma_v$ such that $d(q_k,\bar{q}_k)=d(q_k,\gamma_v)$ and $d(p_k,\bar{p}_k)=d(p_k,\gamma_v)$. Connect $\bar{p}_k$ and $p_k$ by a smooth curve (not necessarily a geodesic) $b_k$ with $b_k(0)=\bar{p}_k$ and $b_k(1)=p_k$	such that $\measuredangle_{\pi(v)}(\gamma_v(\text{--}\infty),b_k(s))$ increases with respect to $s$. Similarly, connect $\bar{q}_k$ and $q_k$ by $c_k$ with a smooth curve $c_{k}(0)=\bar{q}_k, c_{k}(1)=q_k$ and $\measuredangle_{\pi(v)}(\gamma_v(+\infty),c_k(s))$ increases with respect to $s$.

	Denote $\gamma_{k,s}$ the family of geodesics connecting from $b_k(s)$ to $c_k(s)$ with parametrization $\gamma_{k,0}(0)=\gamma_v(0)$ and $s\mapsto\gamma_{k,s}(0)$ is a continuous curve.

	Then $d(\pi(v),\gamma_{k,0}(0))=0$ and $d(\pi(v),\gamma_{k,1}(0))\geq d(\pi(v),\gamma_{p_k,q_k})\geq b$. By Theorem~\ref{thm_3_1}, the continuity of geodesics guarantees $s_k\in(0,1)$ with $d(\pi(v),\gamma_{k,s_k}(0))=b$. Therefore, for any $k\in\mathbb{Z}^+$, $\gamma'_{k,s_k}(0)\in T^1\tilde{M}$ lie in a compact set. By passing to a sub-sequence, let
	\begin{displaymath}
		v'=\lim_{k\to\infty}\gamma'_{k,s_k}(0).
	\end{displaymath}

	From Theorem~\ref{thm_3_1} the argument about the continuity of geodesics at the boundary and the choice of $p_k, q_k$, we have $\measuredangle_{\pi(v)}(\gamma_v(\text{--}\infty),b_k(s))<\frac{1}{k}$ and $\measuredangle_{\pi(v)}(\gamma_v(+\infty),c_k(s))<\frac{1}{k}$.
	\begin{displaymath}
		\begin{aligned}
			\gamma_{v'}(+\infty)&=\lim_{k\to\infty}c_k(s_k)=\gamma_v(+\infty),\\
			\gamma_{v'}(\text{--}\infty)&=\lim_{k\to\infty}b_k(s_k)=\gamma_v(-\infty).
		\end{aligned}
	\end{displaymath}

	Since $\tilde{M}$ has no conjugate points, $b_k(s_k)$ and $c_k(s_k)$ cannot lie on the geodesic $\gamma_v$ simultaneously. Thus $v'$ cannot be on the geodesic $\gamma_v$.

	Therefore, we have two different geodesics $\gamma_{v'}$ and $\gamma_v$ with the same end points on the boundary, contradicting to the assumption that $\textrm{rank}(v)=1$.
\end{proof}

Now prove Theorem~\ref{thm_3_2}.
\begin{proof}[Proof of Theorem~\ref{thm_3_2}]
	Assume $\textrm{rank}(v)=1$. Lemma~\ref{lem_3_2} guarantees some $k_0\in\mathbb{Z}^+$ and an arbitrary constant $b>0$, such that

	\begin{displaymath}
		d(\gamma_v(0),\gamma_{p,q})<b,\quad\forall (p,q)\in C(-v,\frac{1}{k_0})\times C(v,\frac{1}{k_0}).
	\end{displaymath}

	Pick two sequences of converging points $\{p_k{\}}_{k=1}^{\infty}\subseteq C(-v,\frac{1}{k_0})$ and $\{q_k{\}}_{k=1}^{\infty}\subseteq C(v,\frac{1}{k_0})$. Assume that $p_k\to\xi$ and $q_k\to\eta$ in the cone topology.

	Let $\gamma_{p_k,q_k}$ denote the geodesic connecting each pair of $p_k, q_k$ with parametrization such that $d(\gamma_v(0),\gamma_{p_k,q_k}(0))<b$. Similar to the proof of the previous lemma, by passing to a sub-sequence if needed, we can assume that
	\begin{displaymath}
		v'=\lim_{k\to\infty}\gamma'_{p_k,q_k}(0).
	\end{displaymath}

	Using Theorem~\ref{thm_3_1}, we have that
	\begin{displaymath}
		\gamma_{v'}(\text{--}\infty)=\lim_{k\to\infty}p_k=\xi,\quad		\gamma_{v'}(\infty)=\lim_{k\to\infty}q_k=\eta.
	\end{displaymath}

	By the choice of $p_k, q_k$, $\measuredangle_{\gamma_v(0)}(-v, p_k)\leq\frac{1}{k_0}$ and $\measuredangle_{\gamma_v(0)}(v, q_k)\leq\frac{1}{k_0}$. Since $b>0$ is arbitrary, let $b=\epsilon$. We can see $(\xi,\eta)\in U_\epsilon\times V_\epsilon$. And $\gamma_{v'}$ is the connecting geodesic. We can also get $\gamma_{v'}$ close to $\gamma_v$ with $d(\gamma_v(0),\gamma_{v'})<\epsilon$.

	The last thing is to show that this geodesic $\gamma_{v'}$ is rank $1$. Assume not, that is to say, for any $\epsilon_n>0$, we can find converging sequence of points $(\xi_n,\eta_n)\in U_\epsilon\times V_\epsilon$, $\xi_n\to\gamma_v(\text{--}\infty), \eta_n\to\gamma_v(+\infty)$, but the connecting geodesic $\gamma_n=\gamma_{\xi_n,\eta_n}$ has rank at least $2$.

	Parametrize $\gamma_n$ by $d(\gamma_v(0),\gamma_n(0))=d(\gamma_v(0),\gamma_n)$. Thus $\{\gamma_n(0){\}}_{n=1}^\infty\subseteq B(\gamma_v(0),\epsilon)$. By passing to sub-sequence, we can assume $\lim_{n\to\infty}\gamma_n(0)$ exists.

	There are two cases of the limit points.

	If $\lim_{n\to\infty}\gamma_n(0)\neq\gamma_v(0)$, pick a sub-sequence if needed, let $v'=\lim_{n\to\infty}\gamma'_n(0)$. We have $v'\neq \phi_t(v),t\in\mathbb{R}$, but
	\begin{displaymath}
		\gamma_{v'}(\text{--}\infty)=\lim_{n\to\infty}\gamma_n(-\infty)=\gamma_v(-\infty),\quad	\gamma_{v'}(\infty)=\lim_{n\to\infty}\gamma_n(+\infty)=\gamma_v(+\infty).
	\end{displaymath}

	Therefore $\textrm{rank}(v)\geq2$, contradiction.

	If $\lim_{n\to\infty}\gamma_n(0)=\gamma_v(0)$. Pick a sub-sequence if needed, let $\lim_{n\to\infty}\gamma'_n(0)=v$ as they share the same end points. Because the higher rank vectors form a closed set in $T^1\tilde{M}$, we have that $\textrm{rank}(v)\geq 2$, contradiction.
\end{proof}

\section{Local Product Structure and Anosov Closing Lemma}\label{sec4}

In this part we will show that a rank $1$ vector admits local product structure. Moreover, given the manifold satisfies uniform visibility, the Anosov Closing Lemma also holds to the restriction to $\Omega_1$.

\begin{theorem}[Local Product Structure]\label{thm_4_1}
	Let $M$ be a connected, complete manifold with no conjugate points. Suppose $M$ satisfies bounded asymptote. Let $v_0\in T^1M$ be a rank $1$ unit vector and $\epsilon>0$ small. Given any two unit vectors $u,w\in \Omega_1$ close to $v_0$, we can find a rank $1$ unit vector $v$ close to $v_0$, such that $v\in W_{\epsilon}^{ss}(w)\cap W_{\epsilon}^{su}(\phi_t(u))$ for some $|t|<\epsilon$.
\end{theorem}

\begin{proof}
	This is a direct consequence of Theorem~\ref{thm_3_1} and Theorem~\ref{thm_3_2}.

	Given a rank $1$ unit vector $v_0$, lift it to the universal cover $\tilde{v}_0$. Theorem~\ref{thm_3_2} shows that there are two neighborhood $U_{\epsilon/2A}, V_{\epsilon/2A}$ around $\gamma_{v_0}(\text{--}\infty)$ and $\gamma_{v_0}(+\infty)$ on the boundary respectively, such that for any pair of boundary point $(\xi,\eta)\in U_{\epsilon/2A}\times V_{\epsilon/2A}$, there is a rank $1$ connecting geodesic. Here $A$ is the constant given in Lemma~\ref{lem_3_0}.

	Theorem~\ref{thm_3_1} the continuity argument implies that as long as $\tilde{u}, \tilde{w}$ sufficiently close to $\tilde{v}_0$, the end points of $\gamma_{\tilde{u}}$ and $\gamma_{\tilde{w}}$ will be close to the end points of $\gamma_{\tilde{v}_0}$. Therefore, we have that $\gamma_{\tilde{u}}(\text{--}\infty)\in U_{\epsilon/2A}$ and $\gamma_{\tilde{w}}(+\infty)\in V_{\epsilon/2A}$, admits a connecting geodesic $\gamma_{\tilde{v}}$ with $\tilde{v}$ chosen close to $\tilde{v}_0$. Then the projection of $\tilde{v}$ onto the base $M$ gives the rank $1$ unit vector we want.
\end{proof}

\begin{remark}
	From the proof it is not hard to see that the local product structure also holds on the universal cover.

\end{remark}

\begin{theorem}[Rank $1$ Anosov Closing Lemma]\label{thm_4_2}
	Let $M$ be a complete, connected manifold with no conjugate points. Suppose $M$ satisfies bounded asymptote and uniform visibility.

	Given a rank $1$ unit vector $u$, there exists an open neighborhood $U\ni u$ satisfies that for any $\epsilon>0$, there exists $\delta>0$ and $T>0$, such that for any $w\in U$ and $t>T$ with $d_1(w,\phi_t(w))<\delta$, we can find a rank $1$ unit vector $v_0$ and $t_0$ with $|t_0-t|<\epsilon$, $\phi_{t_0}(v_0)=v_0$ and $d_1(\phi_s(v_0),\phi_s(w))<\epsilon$ for any $0<s<\min{\{t,t_0 \}}$.
\end{theorem}

\begin{proof}
	Prove by contradiction. Assume the Anosov Closing Lemma does not hold for $\phi_t$, we can find $u\in U$ and $\epsilon>0$ with a sequence of rank $1$ unit vectors in $v_n\to u$ and an increasing sequence of positive numbers $t_n\to\infty$ such that $d_1(v_n,\phi_{t_n}(v_n))\to 0$. But there are no rank $1$ periodic orbit of $\phi_t$ with period close to $t_n$ that $\epsilon$-shadows the orbit of $v_n$.

	Lift to the universal cover $T^1\tilde{M}$, we can state the previous argument as following: there exists $\epsilon>0$, $\{\tilde{v}_n{\}}_{n=1}^\infty\subseteq T^1\tilde{M}$, $\{t_n{\}}_{n=1}^\infty$, $\{g_n{\}}_{n=1}^\infty\subseteq\Gamma$, such that
	\begin{displaymath}
		\tilde{v}_n\to\tilde{u},~t_n>n\to\infty,~\text{and}~d_1(\tilde{v}_n,\mathrm{d}g_n\phi_{t_n}(\tilde{v}_n))<\frac{1}{n}.
	\end{displaymath}

	But there is no lift of a sequences of periodic geodesics. That is to say, there is no sequence of axis geodesics $\{\gamma_n{\}}_{n=1}^\infty$ with $g_n$ acting as translation on $\gamma_n$ by $g_n\gamma_n(s)=\gamma_n(s+\omega_n)$ where $\omega_n\in(t_n-\epsilon,t_n+\epsilon)$. And in this period, $\gamma_n$ $\epsilon$-shadows $\gamma_{\tilde{v}_n}|_{[0,t_n]}$.

	Now we will show that for sufficiently large $n$, $g_n$ given above is exactly the axial isometry, leading to the contradiction.

	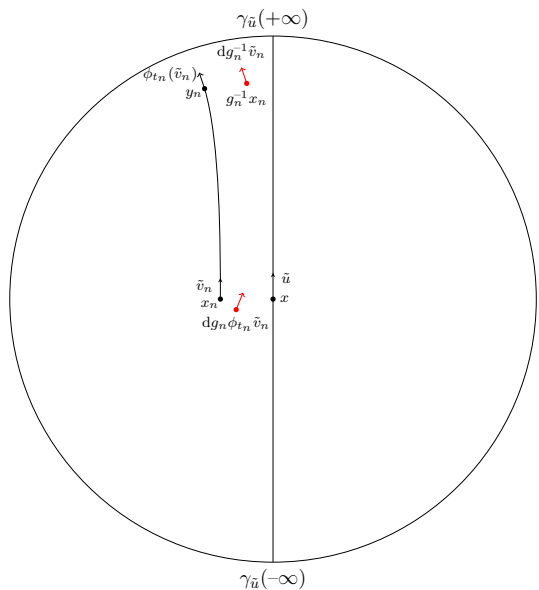
\begin{figure}[htbp]
		\centering
		\scalebox{0.7}{
			\begin{tikzpicture}
				\draw (0,0) circle (5);
				\coordinate [label=right:$\scriptstyle x$] (X) at (0,0);
				\node at (0.25,0.4) {$\scriptstyle \tilde{u}$};
				\node at (X) [circle, fill, inner sep=1pt]{};
				\coordinate [label=above:$\gamma_{\tilde{u}}(+\infty)$] (E) at (0,5);
				\draw [directed] (X) -- (E);
				\coordinate [label=below:$\gamma_{\tilde{u}}(\text{--}\infty)$] (S) at (0,-5);
				\draw (X) -- (S);
				\coordinate [label=above left:$\scriptstyle\tilde{v}_n$] (V) at (-1,0);
				\node at (V) [circle, fill, inner sep=1pt]{};
				\node at (-1.2,-0.1) {$\scriptstyle{x_n}$};
				\coordinate [label=above left:$\scriptstyle\phi_{t_n}(\tilde{v}_n)$] (F) at (-1.3,4);
				\node at (F) [circle, fill, inner sep =1pt]{};
				\node at (-1.5,3.9) {$\scriptstyle y_n$};
				\draw [directed] (V).. controls (-1,1) and (-1,3).. (F);
				\draw [->] (F) -- (-1.4,4.3);
				\coordinate [label=below:$\scriptstyle g^{\text{--}1}_n x_n$] (N) at (-0.5,4.1);
				\node at (N) [red, circle, fill, inner sep=1pt]{};
				\draw [red, ->] (N) -- (-0.6,4.4);
				\node at (-0.6,4.7) {$\scriptstyle \mathrm{d}g_n^{\text{--}1}\tilde{v}_n$};
				\coordinate [label=below:$\scriptstyle\mathrm{d}g_n\phi_{t_n}\tilde{v}_n$] (O) at (-0.7,-0.2);
				\node at (O) [red, circle, fill, inner sep=1pt]{};
				\draw [red, ->] (O) -- (-0.57,0.12);
			\end{tikzpicture}
		}
		\caption{Rank $1$ Anosov Closing Lemma}\label{fig4}
	\end{figure}

	Let $x=\pi(\tilde{u})$, $x_n=\pi(\tilde{v_n})$, $y_n=\pi(\phi_{t_n}(\tilde{v}_n))$ as in the Figure~\ref{fig4}. $g_n$ is an isometry, we have
	\begin{equation}\label{eq_4_1}
		d(y_n,g^{-1}_n x_n)=d(g_n y_{n},x_n)<\frac{1}{n}.
	\end{equation}

	Here $y_n=\gamma_{\tilde{v}_n}(t_n)$, using Theorem~\ref{thm_3_1} the continuity argument, in the cone topology we have the following limit holds

	\begin{equation}\label{eq_4_2}
		\lim_{n\to\infty}y_n=\lim_{n\to\infty}\gamma_{\tilde{v}_n(t_n)}=\gamma_{\tilde{u}}(+\infty).
	\end{equation}

	Equation~\eqref{eq_4_1} and~\eqref{eq_4_2} and together with the cone topology implies that
	\begin{displaymath}
		\lim_{n\to\infty}g^{-1}_n x_n=\gamma_{\tilde{u}}(+\infty).
	\end{displaymath}

	Thus, we have
	\begin{equation}\label{eq_4_3}
		\lim_{n\to\infty}g^{-1}_n x=\gamma_{\tilde{u}}(+\infty).
	\end{equation}

	And similarly, using the method in~\cite{LWW}~(Proposition 4), we can get
	\begin{equation}\label{eq_4_4}
		\lim_{n\to\infty}g_n x=\gamma_{\tilde{u}}(\text{--}\infty).
	\end{equation}

	For the fixed $\epsilon$, choose $V_{\epsilon/2}$ and $U_{\epsilon/2}$ around $\gamma_{\tilde{u}}(\pm\infty)$ as in Theorem~\ref{thm_3_2}.  As the ideal boundary homeomorphic to the unit sphere $\partial\tilde{M}\simeq\mathbb{S}^{n-1}$, we can regard $V_{\epsilon/2}$ and $U_{\epsilon/2}$ open disks in $\mathbb{R}^{n-1}$.

	When $n$ is sufficiently large, Equation~\eqref{eq_4_3} and~\eqref{eq_4_4} imply that
	\begin{displaymath}
		g_n\overline{U}_{\epsilon/2}\subseteq\overline{U}_{\epsilon/2},\quad g^{-1}_n\overline{V}_{\epsilon/2}\subseteq\overline{V}_{\epsilon/2}.
	\end{displaymath}

	Apply the Brouwer Fixed Point Theorem on $\overline{U}_{\epsilon/2}$ and $\overline{V}_{\epsilon/2}$, let $\xi_n\in U_{\epsilon/2}$ and $\eta_n\in V_{\epsilon/2}$ be the fixed points for $g_n$
	\begin{equation}\label{eq_4_5}
		g_n\xi_n=\xi_n,\quad g^{-1}_n\eta_n=\eta_n.
	\end{equation}

	Theorem~\ref{thm_3_2} implies a rank $1$ connecting geodesic $\gamma_n$ with $\gamma_n(\text{--}\infty)=\xi_n$, $\gamma_n(+\infty)=\eta_n$. Parametrize $d(x,\gamma_n(0))=d(x,\gamma_n)$, we have $d(\gamma_{u}(0),\gamma_n)<\epsilon/2$. Let $p_n=\gamma_n(0)$. The Equation~\eqref{eq_4_5} shows that $g_n$ is the axial isometry with the axis $\gamma_n$. That is to say
	\begin{displaymath}
		g_n^{-1}\gamma_n(t)=\gamma_n(t+\omega_n),\quad\forall~t\in\mathbb{R},~\omega_n>0.
	\end{displaymath}

	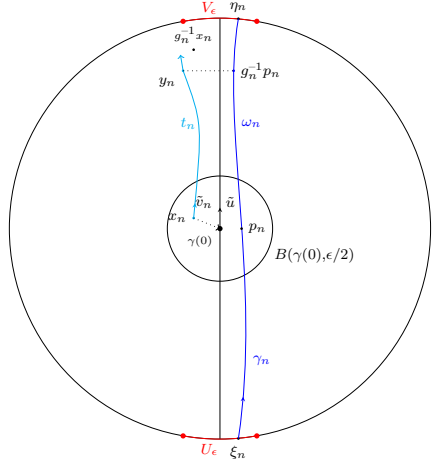
\begin{figure}[htbp]
		\centering
		\scalebox{0.7}{
			\begin{tikzpicture}
				\coordinate [label=below left:$\scriptscriptstyle\gamma(0)$] (O) at (0,0);
				\draw (O) circle (1);
				\draw (O) circle (4);
				\node at (O) [circle, fill, inner sep=1pt]{};
				\node at (1.8,-0.5) {$\scriptstyle B(\gamma(0),\epsilon/2)$};
				\coordinate (E) at (0,4);
				\node at (-0.2,4.2) [red]{$\scriptstyle V_\epsilon$};
				\coordinate (S) at (0,-4);
				\node at (-0.2,-4.2) [red] {$\scriptstyle U_\epsilon$};
				\draw [directed] (O) -- (E);
				\node at (0.2,0.5) {$\scriptstyle\tilde{u}$};
				\draw (S) -- (O);
				\draw[red] (0.7,3.94) arc (80:100:4);
				\draw[red] (-0.7,-3.94) arc (260:280:4);
				\node at (0.7,3.94) [red, circle, fill, inner sep=1pt]{};
				\node at (0.7,-3.94) [red, circle, fill, inner sep=1pt]{};
				\node at (-0.7,3.94) [red, circle, fill, inner sep=1pt]{};
				\node at (-0.7,-3.94) [red, circle, fill, inner sep=1pt]{};

				\coordinate [label=above:$\scriptstyle\eta_n$] (E) at (0.35,3.99);
				\node at (E) [circle, fill, inner sep=0.5pt] {};
				\coordinate [label=below:$\scriptstyle\xi_n$] (X) at (0.35, -3.99);
				\node at (X) [circle, fill, inner sep=0.5pt] {};
				\draw[directed, blue] (X).. controls (0.8,-1) and (0,2).. (E);
				\coordinate [label=right:$\scriptstyle p_n$] (P) at (0.41,0);
				\node at (P) [circle, fill, inner sep=0.5pt] {};
				\node at (0.8,-2.5) [blue] {$\scriptstyle\gamma_n$};
				\coordinate [label=below left:$\scriptstyle y_n$] (Y) at (-0.7,3);
				\node at (Y) [cyan, circle, fill, inner sep=0.5pt] {};
				\draw [cyan, ->] (Y) -- (-0.75,3.3);

				\coordinate [label=above:$\scriptscriptstyle g^{\text{--}1}_nx_n$] (Z) at (-0.5,3.4);
				\node at (Z) [circle, fill, inner sep=0.5pt] {};
				\coordinate [label=right:$\scriptstyle g^{\text{--}1}_np_n$] (Q) at (0.26,3);
				\node at (Q) [blue, circle, fill, inner sep=0.5pt] {};
				\coordinate [label=left:$\scriptstyle x_n$] (T) at (-0.5,0.2);
				\node at (T) [cyan, circle, fill, inner sep=0.5pt]{};
				\draw[directed, cyan] (T).. controls (-0.33,2).. (Y);
				\node at (-0.3,0.5) {$\scriptstyle\tilde{v}_n$};
				\node at (0.6,2) [blue] {$\scriptstyle \omega_n$};
				\node at (-0.6,2) [cyan] {$\scriptstyle t_n$};
				\draw [dotted, ->] (T) -- (O);
				\draw [dotted, thin] (Y) -- (Q);
			\end{tikzpicture}
		}
		\caption{Shadowing by Axis}\label{fig5}
	\end{figure}

	Thus, for sufficiently large $n$, we have
	\begin{displaymath}
		d(g^{-1}_{n}\gamma_n (0),g^{-1}_n x_n)=d(\gamma_n(0),x_n)\leq d(\gamma_n(0),x)+d(x,x_n)\leq\epsilon/2+1/n<\epsilon.
	\end{displaymath}

	We can estimate $\omega_n$ by:
	\begin{displaymath}
		\begin{aligned}
			\omega_n&=d(\gamma_n(0),g^{-1}_n\gamma_n(0))\\
					&\leq d(\gamma_n(0)+x_n)+d(x_n,y_n)+d(y_n,g^{-1}_n\gamma_n(0))\\
					&\leq\epsilon+t_n+d(y_n,g^{-1}_n x_n)+d(g^{-1}_n x_n,g^{-1}_n\gamma_n(0))\\
					&\leq\epsilon+t_n+\frac{1}{n}+\frac{1}{n}+\epsilon/2=t_n+2\epsilon.
		\end{aligned}
	\end{displaymath}

	Similarly, by replacing $g_n$ by $g^{-1}_n$ we have $\omega_n\geq t_n-2\epsilon$.

	Since rank $1$ vectors form an open set, for sufficiently large $n$, $\gamma_n'(0)$ is very close to the rank $1$ vector $\tilde{u}$, thus it is also rank $1$ unit vector. Only need to show that on $\tilde{M}$, $\gamma_n|_{[0,\omega_n]}$~$\epsilon$-shadows $\gamma_{\tilde{v}_n}|_{[0,t_n]}$, which leads to a rank $1$ periodic orbit $\epsilon$-shadows $\gamma_{v_n}|_{[0,t_n]}$ on $M$ as we want.

	This result can be easily achieved on manifolds with non-positive curvature or no focal points, but much harder on manifolds with no conjugate points and bounded asymptote. Uniform visibility is needed in our proof here.

	Prove this by contradiction. Assume the conclusion fails. That is to say we have the distance of the end points are controlled by any given $\epsilon$: $d(\gamma_{v_n}(0),\gamma_n(0))<\epsilon$, $d(\gamma_{v_n}(t_n),\gamma_n(\omega_n))<\epsilon$. But the maximal distance of $\gamma_n$ and $\gamma_{\tilde{v}_n}$ is not controlled by $\epsilon$.
	\begin{equation}\label{eq_4_9}
		\max_{0\leq t\leq\min{\{t_n,\omega_n\}}} d(\gamma_n(t),\gamma_{\tilde{v}_n}(t))>\epsilon.
	\end{equation}

	On $\tilde{M}$, see Figure~\ref{fig5}
	\begin{displaymath}
		d(g_n^{-1}x_n,\gamma_{\tilde{v}_n}(t_n))<\frac{1}{n},\quad d_n(g^{-1}_n p_n,\gamma_{\tilde{v}_n}(t_n))\leq \frac{2}{n}.
	\end{displaymath}

	Using Theorem~\ref{thm_3_1} the continuity argument, we have that

	\begin{equation}\label{eq_4_6}
		\gamma_{\tilde{u}}(+\infty)=\lim_{n\to\infty}\gamma_{\tilde{v}_n}(t_n)=\lim_{n\to\infty}g^{-1}_n x_n=\lim_{n\to\infty}g^{-1}_n p_n=\lim_{n\to\infty}\gamma_n(\omega_n).
	\end{equation}

	Similarly, we have that
	\begin{equation}\label{eq_4_7}
		\gamma_{\tilde{u}}(\text{--}\infty)=\lim_{n\to\infty}\gamma_n(-\omega_n).
	\end{equation}

	By passing to a sub-sequence if needed, denote that
	\begin{displaymath}
		\tilde{v}=\lim_{n\to\infty}\gamma'_n(0),\quad p=\lim_{n\to\infty}\gamma_n(0)=\gamma_{\tilde{v}}(0).
	\end{displaymath}

	Equation~\eqref{eq_4_6} and~\eqref{eq_4_7} implies that these two geodesics are bi-asymptotic.
	\begin{equation}\label{eq_4_8}
		\gamma_{\tilde{v}}(\pm\infty)=\gamma_{\tilde{u}}(\pm\infty).
	\end{equation}

	Because the rank of $\tilde{u}$ is $1$, $p\in\gamma_{\tilde{u}}$. Otherwise, bi-asymptote of the two geodesics $\gamma_{\tilde{u}}$ and $\gamma_{\tilde{v}}$ implies that the rank is greater than $1$, contradiction.

	As $p=\lim_{n\to\infty}p_n$. By the choice of $p_n$ we have $d(\gamma_{\tilde{u}}(0),p_n)=d(\gamma_{\tilde{u}}(0),\gamma_n)$. Theorem~\ref{thm_3_2} implies that the connecting geodesic $\gamma_n(0)$ can be chosen close to $\gamma_{\tilde{u}}$. Passing to a sub-sequence if needed, we can assume $d(p_n,x)\leq\frac{1}{n}$. Therefore $p=\gamma_{\tilde{u}}(0)=x$.

	Without loss of generality, assume $t_n>n$ and $\omega_n>n$,
	\begin{displaymath}
		\begin{aligned}
			d(\gamma_{\tilde{v}_n}(0),\gamma_n(0))&=d(x_n,p_n)\leq d(x_n,x)+d(x,p_n)\leq\frac{2}{n},\\
			d(\gamma_{\tilde{v}_n}(t_n),\gamma_n(\omega_n))&\leq d(\gamma_{\tilde{v}_n}(t_n),g^{-1}_n\gamma_{\tilde{v}_n}(0))+d(g^{-1}_n\gamma_{\tilde{v}_n}(0),g^{-1}_n\gamma_n(0))\leq\frac{3}{n}.
		\end{aligned}
	\end{displaymath}

	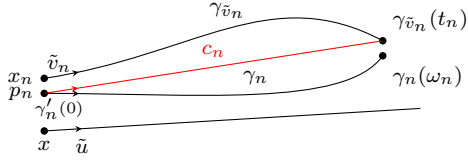
\begin{figure}[htbp]
		\centering
		\scalebox{1}{
			\begin{tikzpicture}
				\coordinate [label=below:$\scriptstyle x$] (X) at (-2,-2);
				\node at (X) [circle, fill, inner sep=1pt]{};					\draw[directed] (X)--(3,-1.7) node [pos=0.1,below] {$\scriptstyle \tilde{u}$};
				\coordinate [label=left:$\scriptstyle p_n$] (P) at (-2,-1.5);
				\node at (P) [circle, fill, inner sep=1pt]{};
				\coordinate [label=left:$\scriptstyle x_n$] (O) at (-2,-1.3);
				\node at (O) [circle, fill, inner sep=1pt]{};
				\coordinate [label=below right:$\scriptstyle \gamma_{n}(\omega_n)$] (U) at (2.5, -1);
				\node at (U) [circle, fill, inner sep=1pt]{};
				\coordinate [label=above right:$\scriptstyle \gamma_{\tilde{v}_n}(t_n)$] (V) at (2.5, -0.8);
				\node at (V) [circle, fill, inner sep=1pt]{};

				\draw [red, directed] (P)--(V) node [midway, above] {$\scriptstyle c_n$};

				\draw[directed] (O).. controls (0, -1) and (1,0).. (V) node [midway, above] {$\scriptstyle \gamma_{\tilde{v}_n}$};
				\node at (-1.8,-1.1) {$\scriptstyle\tilde{v}_n$};
				\draw[directed] (P).. controls (0,-1.5) and (2,-1.7).. (U) node [midway, above] {$\scriptstyle \gamma_n$};
				\node at (-1.8,-1.7) {$\scriptscriptstyle\gamma'_n(0)$};
			\end{tikzpicture}
		}
		\caption{Non-Shadowing}\label{fig6}
	\end{figure}

	Let $c_n=\gamma_{p_n,\gamma_{\tilde{v}_n}(t_n)}$ denote the connecting geodesic and $s_n$ denote the length of this segment. By triangular inequality,
	\begin{displaymath}
		|s_n-\omega_n|\leq\frac{3}{n},\quad |s_n-t_n|\leq\frac{2}{n}.
	\end{displaymath}

	In Equation~\eqref{eq_4_9}, assume there exists $t_{n,0}\in[0,\min{(t_n,\omega_n)}]$ such that $d(\gamma_n(t_{n,0}),\gamma_{\tilde{v}_n}(t_{n,0}))>\epsilon$. Then $d(c_n(t_{n,0}),\gamma_n(t_{n,0}))+d(c_n(t_{n,0}),\gamma_{\tilde{v}_n}(t_{n,0}))>\epsilon$. Without loss of generality, assume that
	\begin{displaymath}
		d(c_n(t_{n,0}),\gamma_n(t_{n,0}))>\frac{\epsilon}{2}.
	\end{displaymath}

	\begin{center}
		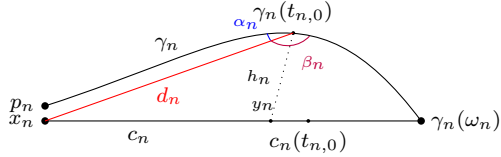
\begin{figure}[htbp]
			\centering
			\scalebox{1}{
				\begin{tikzpicture}
					\coordinate [label=left:$\scriptstyle x_n$] (X) at (-1.5,0);
					\node at (X) [circle, fill, inner sep=1pt] {};
					\coordinate [label=left:$\scriptstyle p_n$] (P) at (-1.5,0.2);
					\node at (P) [circle, fill, inner sep=1pt] {};
					\coordinate [label=right:$\scriptstyle \gamma_n(\omega_n)$] (Q) at (3.5,0);
					\node at (Q) [circle, fill, inner sep=1pt] {};
					\draw (X)--(Q) node [near start, below] {$\scriptstyle c_n$};
					\coordinate [label=below:$\scriptstyle c_n(t_{n,0})$] (A) at (2,0);
					\node at (A) [circle, fill, inner sep=0.5pt] {};
					\draw (P).. controls (1,1) and (2,2).. (Q) node [near start, above] {$\scriptstyle \gamma_n$};
					\coordinate [label=above:$\scriptstyle \gamma_n(t_{n,0})$] (B) at (1.8,1.17);
					\node at (B) [circle, fill, inner sep=0.5pt] {};
					\draw[red] (X)--(B) node [midway, below] {$\scriptstyle d_n$};
					\coordinate (C) at (1.5,1.07);
					\draw [blue] (C) arc (215:198:0.36) node [near start, above left] {$\scriptscriptstyle \alpha_n$};
					\draw [purple] (C) arc (236:317:0.4) node [midway, below right] {$\scriptscriptstyle \beta_n$};
					\node at (1.5,0) [circle, fill, inner sep=0.5pt] {};
					\node at (1.4,0.2) {$\scriptscriptstyle y_n$};
					\draw [dotted] (1.5,0)--(B) node [midway, left] {$\scriptscriptstyle h_n$};
				\end{tikzpicture}
			}
			\caption{$\epsilon$-Separation}\label{fig7}
		\end{figure}
	\end{center}

	Let $d_n$ denote the geodesic segment connecting $x_n$ and $q_n=\gamma_n(t_{n,0})$. $\alpha_n$ and $\beta_n$ be the angle between geodesic as labeled in Figure~\ref{fig7}. Choose $y_n\in c_n$ such that
	\begin{displaymath}
		d(q_n,y_n)=d(q_n,c_n)=h_n.
	\end{displaymath}

	By triangular inequality, we have that
	\begin{equation}
		2t_{n,0}+d(p_n,x_n)\geq d(q_n,c_n(t_{n,0}))>\frac{\epsilon}{2}.
	\end{equation}

	Since $d(p_n,x_n)<\frac{1}{n}$. When $n$ is sufficiently large, we have that
	\begin{equation}
		t_{n,0}>\frac{\epsilon}{4}-\frac{1}{n}>\frac{\epsilon}{8}.
	\end{equation}

	The length of geodesic segment between $p_n$ and $q_n$ is $t_{n,0}$. The difference of length between this segment and the geodesic segment $d_n$ between $x_n$ and $\gamma_n(t_{n,0})$ is at most $\frac{2}{n}$ by triangular inequality.

	Ruggiero~\cite{Ru1} proved that there exists a constant $C$ only depends on the manifold $M$, such that the distance between the geodesic segment is controlled by the angle of the two geodesics $d(\gamma_{[q_n,p_n]},\gamma_{[q_n,x_n]})\geq C\alpha_n t_{n,0}$. Thus
	\begin{equation}
		\alpha_n\leq \frac{1}{Ct_{n,0}}\frac{2}{n}\leq\frac{32}{Cn\epsilon}.
	\end{equation}

	Here $\epsilon$ is some fixed number through the proof. Then $\alpha_n\to0$ as $n\to\infty$. And $\alpha_n+\beta_n=\pi$, leading to that $\beta_n\to\pi$.

	Because $M$ admits uniform visibility, $\beta_n\to\pi$ means that the distance between $q$ and $c_n$ goes to $0$. So $h_n\to0$. Now consider the distance between $y_n$ and $c_n(t_{n,0})$. We have that
	\begin{equation}
		\begin{aligned}
			d(y_n,c_n(t_{n,0}))&=t_{n,0}-d(x_n,y_n)\\
							   &\leq t_{n,0}-d(x_n,q)+h_n\\
							   &\leq \frac{2}{n}+h_n.
		\end{aligned}
	\end{equation}

We now deduce a contradiction for the fixed $\epsilon$.
\begin{displaymath}
	\frac{\epsilon}{2}<d(q_n,c_n(t_{n,0}))\leq d(y_n,c_n(t_{n,0}))+h_n\leq \frac{2}{n}+2h_n\to 0.
\end{displaymath}

This completes the proof.
\end{proof}

A straightforward consequence of this theorem is the Anosov Closing Lemma restricted to $\Omega_1$ on $T^1\tilde{M}$.

\begin{theorem}[Anosov Closing Lemma on $\Omega_1$]\label{thm_4_3}
	Let $M$ be the manifold as stated in Theorem~\ref{thm_4_2}, then Anosov Closing Lemma holds to the restriction to $\Omega_1$.

	That is to say, given any $\epsilon>0$, there exist $\delta>0$ and $T>0$, such that for any $v\in\Omega_1$ and $t>T$ with $d_1(v,\phi_t(v))<\delta$, we can find a rank $1$ periodic unit vector $v_0\in\Omega_1$ and $t_0$ with $|t_0-t|<\epsilon$, $\phi_{t_0}(v_0)=v_0$ and $d_1(\phi_s(v_0),\phi_s(v))<\epsilon$ for any $0<s<\min{\{t,t_0\}}$.
\end{theorem}

\begin{proof}
	Notice that the rank $1$ unit vectors form an open set in $T^1 M$. From Theorem~\ref{thm_4_2}, we know that the rank $1$ vector $v$ can be shadowed by a rank $1$ periodic vector $v_0$. And $v_0$ is periodic, thus recurrent. Proposition~\ref{thm_3_3} implies that $v_0$ is contained in $\Omega_1$.
\end{proof}

\begin{remark}
	We can easily adapt our argument to the restriction of $\Omega_{\text{NF}}$ with an extra assumption that $\Omega_{\text{NF}}$ is relatively open in $\Omega$. The proof is similar to the proof given here. Thus, we can claim that the Anosov Closing Lemma holds to the restriction of $\Omega_{\text{NF}}$.
\end{remark}

\section{\bf Topological Transitivity}\label{sec5}
In this part, we will prove the transitivity property of the geodesic flow to the restriction to $\Omega_1$.

\begin{theorem}[Topological Transitivity]\label{thm_5_0}
	Let $M$ be a connected, complete manifolds with no conjugate points. Suppose $M$ satisfies bounded asymptote and uniform visibility, and the geodesic flows has at least three periodic orbits in $\Omega_1$. Then the geodesic flow restricted to $\Omega_1$ is topological transitive.
\end{theorem}

Some results are needed in preparation to prove this theorem. We start with the case on surfaces.

\begin{lemma}\label{lem_5_1}
	Let $M$ be a Riemannian surface with no conjugate points and genus no less than $2$. Let $\tilde{M}$ denote the universal cover of $M$.  Suppose $\gamma$ is the axis of the axial isometry $g\in\Gamma\subseteq \textrm{Iso}(\tilde{M})$, $\gamma(0)=p$. $x\in\partial\tilde{M}-{\gamma(+\infty),\gamma(\text{--}\infty)}$,then $gx$ lies between $x$ and $g^2x$.
\end{lemma}

\begin{proof}
	Take $p=\gamma(0)$. As $\gamma$ is the axis of $g$, which acts as translation on $\gamma$. Pick a point $x$ on the boundary and consider the image under $g$ and $g^2$.

	Assume the argument fails. Because the ideal boundary is homeomorphic to the circle $\mathbb{S}^1$, there are only two possibilities: $gx$ above $x$ while $g^2x$ below it, or $gx$ below $x$ while $g^2x$ above it. Without loss of generality, consider the latter case as showing in Figure~\ref{fig8}.

	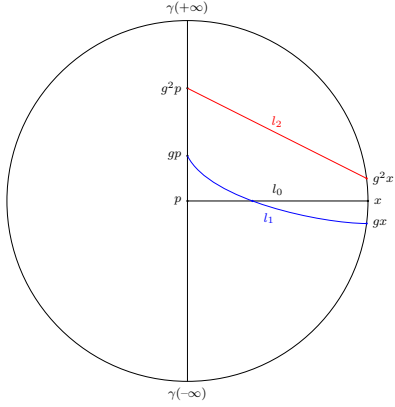
\begin{figure}[htbp]
		\centering
		\scalebox{0.6}{
			\begin{tikzpicture}
				\coordinate [label=left:$\scriptstyle p$] (O) at (0,0);
				\draw (0,0)	circle (4);
				\node at (O) [circle, fill, inner sep =0.5pt]{};
				\coordinate [label=left:$\scriptstyle gp$] (A) at (0,1);
				\coordinate [label=left:$\scriptstyle g^2p$] (B) at (0,2.5);
				\coordinate [label=below:$\scriptstyle \gamma(\text{--}\infty)$] (S) at (0,-4);
				\coordinate [label=above:$\scriptstyle \gamma(+\infty)$] (E) at (0,4);
				\node at (A) [circle, fill, inner sep=0.5pt] {};
				\node at (B) [circle, fill, inner sep=0.5pt] {};
				\draw (S)--(E);

				\coordinate [label=right:$\scriptstyle x$] (P) at (4,0);
				\node at (P) [circle, fill, inner sep=0.5pt]{};
				\coordinate [blue, label=right:$\scriptstyle gx$] (Q) at (3.97,-0.5);
				\node at (Q) [blue, circle, fill, inner sep=0.5pt]{};
				\coordinate [red, label=right:$\scriptstyle g^2x$] (R) at (3.97,0.5);
				\node at (R) [red, circle, fill, inner sep=0.5pt]{};
				\draw (O)--(P) node [midway, above] {$\scriptstyle l_0$};
				\draw [blue] (A).. controls (0.5,0) and (3,-0.5).. (Q) node [midway, below,blue] {$\scriptstyle l_1$};
				\draw [red] (B)--(R) node [midway, above, red] {$\scriptstyle l_2$};
				\node at (1.44,0) [blue, circle, fill, inner sep=0.5pt]{};
			\end{tikzpicture}
		}
		\caption{Order on the Boundary of Surfaces}\label{fig8}
	\end{figure}

	On the surface, due to the limitation of the dimension, we have that $l_0\cap l_1\neq\emptyset$ but $l_2\cap l_0=\emptyset$, $l_1\cap l_2=\emptyset$. On the other hand, $l_2\cap l_1=gl_1\cap gl_0=g(l_1\cap l_0)\neq\emptyset$, contradiction. $gx$ must lie between $x$ and $g^2x$.
\end{proof}

\begin{remark}
	Actually $\gamma(\pm\infty)$ are the only two fixed points of $g$. We know that $gx\neq x$ for $x\neq\gamma(\pm\infty)$. Thus, the sequence of points $\{g^k x{\}}_{k=\text{--}\infty}^\infty$ has an order on the circle. This lemma is a well-known result from dynamics that homeomorphism on $\mathbb{S}^1$ is either keeping the direction or reverse the direction between fixed points.
\end{remark}

\begin{lemma}[Surfaces]\label{lem_5_2}
	Let $M$ be a Riemannian surface with no conjugate points, bounded asymptote and the genus no less than $2$. $\tilde{M}$ is the universal cover of $M$. Suppose $\gamma$ is a rank $1$ axis geodesic on $\tilde{M}$ of axial isometry $g\in\Gamma\subset \textrm{Iso}(\tilde{M})$, for any open neighborhood $U\ni \gamma(\text{--}\infty)$ and $V\ni \gamma(+\infty)$ in $\partial\tilde{M}$, there exists $n_0\in\mathbb{Z}^+$ such that
	\begin{displaymath}
		g^n(\partial\tilde{M}-U)\subseteq V, \quad g^{-n}(\partial\tilde{M}-V)\subseteq U,\quad\forall n\geq n_0.
	\end{displaymath}
\end{lemma}
\begin{proof}
	Prove by contradiction. Assume not, we can find open neighborhood $U_0\ni\gamma(\text{--}\infty)$ and $V_0\ni\gamma(+\infty)$ in $\partial\tilde{M}$, and an increasing sequence $\{n_k{\}}_{k=1}^\infty\subset\mathbb{Z}^+$, such that $g^{n_k}(\partial\tilde{M}-U_0)\cap(\partial\tilde{M}-V_0)\neq\emptyset$.

	That is to say that $g^{n_k}$ does not send the complement of $U_0$ into $V_0$. We can find $\{x_k{\}}_{k=1}^\infty\subset\partial\tilde{M}-U_0$ satisfying that
	\begin{displaymath}
		g^{n_k}x_k\in\partial\tilde{M}-V_0.
	\end{displaymath}

	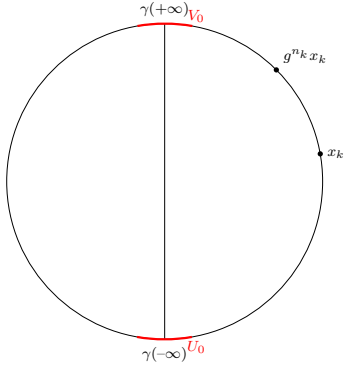
\begin{figure}[htbp]
		\centering
		\scalebox{0.7}{
			\begin{tikzpicture}
				\draw (0,0) circle (3);
				\coordinate[label=above:$\scriptstyle\gamma(+\infty)$] (E) at (0,3);
				\coordinate[label=below:$\scriptstyle\gamma(\text{--}\infty)$] (S) at (0,-3);
				\draw (S)--(E);
				\draw[very thick, red] (0.525,2.955) arc (80:100:3);
				\draw[very thick, red] (-0.525,-2.955) arc (260:280:3);
				\node at (0.6,3.15) [red] {$\scriptstyle V_0$};
				\node at (0.6,-3.15)[red] {$\scriptstyle U_0$};
				\coordinate [label=right:$\scriptstyle x_k$] (X) at (2.955,0.525);
				\node at (X) [fill, circle, inner sep=1pt]{};
				\coordinate [label=above right:$\scriptstyle g^{n_k}x_k$] (Y) at (2.121,2.121);
				\node at (Y) [fill, circle, inner sep=1pt]{};
			\end{tikzpicture}
		}
		\caption{Counter-Clockwise}\label{fig9}
	\end{figure}

	Without loss of generality, we can assume that  $\forall k\in\mathbb{Z}^+, g^{n_k}$ maps $x_k$ counter-clockwise as shown in Figure~\ref{fig9}. Lemma~\ref{lem_5_1} guarantees all $g^i x_k$ lies between $x_k$ and $g^{n_k}x_k$ for $i=1,2,\ldots n_k-1$. As we know that $n_k$ is an increasing sequence of integers, we can assume $n_k\geq k$. Therefore $\forall x$ lies between $x_k$ and $\gamma(\text{--}\infty)$ on the boundary, $g^k x$ must lies between $\gamma(-\infty)$ and $g^k x_k$, thus below $g^{n_k}x_k$.

	Consider the collection of all $x_k$. By the choice of $x_k\in \partial\tilde{M}-U_0$, $\gamma(\text{--}\infty)$ cannot be an accumulative point of $\{x_k\}$. Shrink $U_0$ if needed, we can find a point $x\in \partial\tilde{M}-U_0$ lies between $\gamma(-\infty)$ and all $x_k$ which is an accumulate point of $x_k$. This implies that $g^n x$ lies between $\gamma(-\infty)$ and $g^{n_k}x_k$, thus we have
	\begin{displaymath}
		g^k x\in\partial\tilde{M}-V_0,\qquad\forall n\in\mathbb{Z}^+.
	\end{displaymath}

	By passing to a sub-sequence if needed, we can assume that $\lim_{i\to\infty} g^{n_i}x=y\in\partial\tilde{M}-V_0$.

	\begin{figure}[htbp]
		\centering
		\scalebox{0.7}{
			\begin{tikzpicture}
				\draw (0,0) circle (4);
				\coordinate[label=above:$\scriptstyle\gamma(+\infty)$] (E) at (0,4);
				\coordinate[label=below:$\scriptstyle\gamma(\text{--}\infty)$] (S) at (0,-4);
				\draw (S)--(E);
				\draw[very thick, red] (0.7,3.94) arc (80:100:4);
				\draw[very thick, red] (-0.7,-3.94) arc (260:280:4);
				\node at (0.7,4.2) [red] {$\scriptstyle V_0$};
				\node at (0.7,-4.2)[red] {$\scriptstyle U_0$};
				\coordinate [label=right:$\scriptstyle y$] (Y) at (2.828,2.828);
				\node at (Y) [circle, fill, inner sep=1pt] {};
				\coordinate [label=right:$\scriptstyle x$] (X) at (2.828,-2.828);
				\node at (X) [circle, fill, inner sep=1pt] {};

				\coordinate [label=right:$\scriptstyle g^{n_i}x$] (Z) at (4,0);
				\node at (Z) [circle, fill, inner sep=1pt] {};

				\draw [directed, red] (S).. controls (0,-2) and (2,1).. (Y) node [midway, above] {$\scriptstyle \gamma'$};

				\draw [directed] (S).. controls (0,-2.5) and (2,0).. (Z) node [midway, below] {$\scriptstyle \qquad g^{n_i}C$};
				\draw [directed] (S).. controls (0,-3.5) and (2,-3).. (X) node [midway, below] {$\scriptstyle C$};
				\coordinate [label=left:$\scriptstyle p$] (P) at (0,0);
				\node at (P) [circle, fill, inner sep=1pt] {};
				\coordinate [label=left:$\scriptstyle g^{n_i}p$] (Q) at (0,3.5);
				\node at (Q) [circle, fill, inner sep=1pt] {};

				\draw [directed] (P).. controls (1,-2).. (X);
				\draw [directed] (Q).. controls (1, 1).. (Z);
			\end{tikzpicture}
		}
		\caption{Lemma~\ref{lem_5_2}}\label{fig10}
	\end{figure}

	Theorem~\ref{thm_2_2} tells that such $\tilde{M}$ satisfying the uniform visibility. By Theorem~\ref{thm_2_1}, there exists a connecting geodesic between any two different points on the boundary.

	Let $C$ denote the geodesic connecting $\gamma(\text{--}\infty)$ and $x$, and $\gamma'$ the geodesic connecting $\gamma(-\infty)$ and $y$. Pick any point $p\in\gamma$. Since $g\in \textrm{Iso}(\tilde{M})$ and $\gamma(-\infty)\in \textrm{Fix}(g)$, it keeps the angle

	\begin{displaymath}
		\measuredangle_{g^{n_i}p}(\gamma(\text{--}\infty),g^{n_i}x)=\measuredangle_p(\gamma(-\infty),x)>0.
	\end{displaymath}

	On the other hand, on surfaces, the distance between point $g^{n_i}p$ to the geodesic $g^{n_i}C$ is greater than the distance to $\gamma'$. We have

	\begin{displaymath}
		d(g^{n_i}p, g^{n_i}C)\geq(g^{n_i}p,\gamma')\to\infty.
	\end{displaymath}

	But uniform visibility requests the angle need to tend to $0$ when the distance tend to $\infty$, leading to a contradiction that $\measuredangle_{g^{n_i}p}(\gamma(\text{--}\infty),g^{n_i}x)$ is a constant.
\end{proof}

As shown in the proof, the core of this proof is the argument using the uniform visibility. Inspired by the work in~\cite{CKW}, we are able to generalize this lemma to manifolds with uniform visibility but remove the requirement of the dimension.

\begin{lemma}[General Case]\label{lem_5_3}
	Let $M$ be a Riemannian manifold with no conjugate points and $\tilde{M}$ be its universal cover. Suppose that $M$ satisfies bounded asymptote and uniform visibility. Let $\gamma$ denote an axis geodesic on $\tilde{M}$ of the axial isometry $g\in\Gamma\subset \textrm{Iso}(\tilde{M})$. For any open neighborhoods $U\ni\gamma(\text{--}\infty)$ and $V\ni\gamma(+\infty)$ on the boundary, there exists $n_0\in\mathbb{Z}^+$ such that
	\begin{displaymath}
		g^n(\partial\tilde{M}-U)\subseteq V,\quad g^{-n}(\partial\tilde{M}-V)\subseteq U,\quad\forall n\geq n_0.
	\end{displaymath}
\end{lemma}
\begin{proof}
	Given $\epsilon>0$, let $R_\epsilon=R(\epsilon)$ be the constant in the definition of uniform visibility, and $p=\gamma(0)$.

	Assume not. Parallel to the $2$-dimensional case, we can find open neighborhood $U_0\ni\gamma(\text{--}\infty)$ and $V_0\ni\gamma(+\infty)$ in $\partial\tilde{M}$, and an increasing sequence $\{n_k{\}}_{k=1}^\infty\subset\mathbb{Z}^+$, such that $g^{n_k}(\partial\tilde{M}-U_0)\cap(\partial\tilde{M}-V_0)\neq\emptyset$. Therefore, there is a sequence of points $\{x_k{\}}_{k=1}^\infty\subset\partial\tilde{M}-U_0$ with $g^{n_k}x_k\in\partial\tilde{M}-V_0$.

	Using the cone topology on $\tilde{M}$, there exist $\epsilon_0>0$ such that
	\begin{displaymath}
		\begin{aligned}
			\{\xi\in\partial{M}\mid\measuredangle_p(\gamma(\text{--}\infty),\xi)\leq\epsilon_0\}&\subset U_0,\\
			\{\eta\in\partial{M}\mid\measuredangle_p(\gamma(+\infty),\eta)\leq\epsilon_0\}&\subset V_0.\\
		\end{aligned}
	\end{displaymath}

	Let $R=R(\epsilon_0)$ and $a>0$ be the distance of translation of $g$ on $\gamma$: $g(\gamma(t))=\gamma(t+a)$.

	Consider $x_k\in\partial\tilde{M}-U_0$, $\alpha=\measuredangle_p(x_k,\gamma(\text{--}\infty))\geq\epsilon_0$. Since $g$ keeps the angle and $\gamma(\pm\infty)\in \textrm{Fix}(g)$, we have that
	\begin{displaymath}
		\measuredangle_{g^{n_k}(p)}(g^{n_k}x_k,\gamma(\text{--}\infty))=\measuredangle_p(x_k,\gamma(-\infty))=\alpha\geq\epsilon_0.
	\end{displaymath}

	\begin{figure}[htbp]
		\centering
		\scalebox{0.7}{
			\begin{tikzpicture}
				\coordinate [label=left:$\scriptstyle p$] (P) at (0,0);
				\node at (P) [circle, fill, inner sep=1pt] {};
				\draw (P) circle (4);
				\coordinate[label=above:$\scriptstyle\gamma(+\infty)$] (E) at (0,4);
				\coordinate[label=below:$\scriptstyle\gamma(\text{--}\infty)$] (S) at (0,-4);
				\draw (S)--(E);
				\draw[very thick, red] (0.7,3.94) arc (80:100:4);
				\draw[very thick, red] (-0.7,-3.94) arc (260:280:4);
				\node at (0.7,4.2)  [red] {$\scriptstyle V_0$};
				\node at (0.7,-4.2) [red] {$\scriptstyle U_0$};
				\coordinate [label=left:$\scriptstyle g^{n_k}p$] (R) at (0,2.5);
				\node at (R) [circle, fill, inner sep=1pt] {};
				\coordinate[label=right:$\scriptstyle g^{n_k}x_k$] (Y) at (2.828,2.828);
				\node at (Y) [circle, fill, inner sep=1pt] {};
				\coordinate[label=right:$\scriptstyle x_k$](X) at (2.828,-2.828);
				\node at (X) [circle, fill, inner sep=1pt] {};
				\coordinate [label=above:$\scriptstyle q_k$] (Q) at (1.5,2.55);
				\node at (Q) [circle, fill, inner sep=1pt] {};
				\coordinate [blue, label=right:$\scriptstyle z_k$] (Z) at (1.03,1.8);
				\node at (Z) [blue, circle, fill, inner sep=1pt] {};
				\draw [directed] (P).. controls (1,-0.3) and (2,-1).. (X);
				\draw [purple] (0,-0.5) arc (270:340:0.5);
				\draw [directed] (P).. controls (1,1).. (Y);
				\draw [directed] (R).. controls (1,2.4) and (2,2.6).. (Y);
				\draw [purple] (0,2) arc (270:355:0.5);
				\draw [directed, blue] (P).. controls (0.5,1)..(Q);
				\draw [blue] (R)--(Z);
				\node at (0.3,-0.7) {$\scriptstyle\alpha$};
			\end{tikzpicture}
		}
		\caption{Contraction to the End Point}\label{fig11}
	\end{figure}
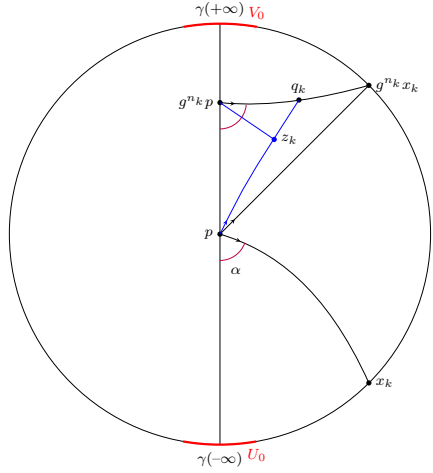

	As shown in Figure~\ref{fig11}, take any point $q_k$ on the geodesic ray $\gamma_{g^{n_k}p,g^{n_k}x_k}$, we have that $\measuredangle_{g^{n_k}p}(\gamma(\text{--}\infty),q_k)\geq\epsilon_0$.

	The uniform visibility implies that the distance between the point $g^{n_k}p$ and the geodesic segment $\gamma_{p,q_k}$ need to be no greater than the uniform constant $R$, say $d(g^{n_k}p, \gamma_{p,q_k})\leq R$.

	Thus, we can find a point $z_k$ on the geodesic segment $\gamma_{p,q_k}$ such that $d(g^{n_k}p,z_k)\leq R$. By the triangular inequality $d(p,z_k)\geq d(p,g^{n_k}p)-d(g^{n_k}p,z_k)\geq n_k a-R$.

	Since $n_k\to +\infty$, we can find some $K>0$, such that $n_k a\geq 3R$, $\forall k>K$. We have
	\begin{equation}
		d(p,z_k)\geq n_k a-R\geq 2R, \quad\forall k\geq K.
	\end{equation}

	By the choice of $z_k$, we know that $d(p,q_k)\geq d(p,z_k)\geq 2R$, which implies that the distance from point $p$ to the geodesic $\gamma_{g^{n_k}p,g^{n_k}x_k}$ is no less than $2R$.

	Again, because of the uniform visibility of the manifold,
	\begin{displaymath}
		\measuredangle_p(g^{n_k}p, g^{n_k}x_k)=\measuredangle_p(\gamma(+\infty), g^{n_k}x_k)\leq\epsilon_0,\quad\forall k\geq K.
	\end{displaymath}

	This implies that $g^{n_k}x_k\in V_0$, contradiction.
\end{proof}

And the last lemma needed is the following.
\begin{lemma}\label{lem_5_4}
	Let $M$ is a Riemannian manifold with no conjugate points. Suppose $M$ satisfies bounded asymptote and uniform visibility. $\tilde{M}$ is the universal cover. Let $v$ be a rank $1$ periodic vector on $T^1M$ and $\tilde{v}$ be its lift to $T^1\tilde{M}$.

	Let $\xi$ denote $\gamma_{\tilde{v}}(+\infty)$ on the boundary, then for any $\eta\neq\xi$ on the boundary, there exist a unique rank $1$ geodesic connecting $\xi$ and $\eta$. That is to say, there exist $\tilde{u}\in T^1\tilde{M}$ such that
	\begin{displaymath}
		\gamma_{\tilde{u}}(\text{--}\infty)=\eta,\quad\gamma_{\tilde{u}}(+\infty)=\xi.
	\end{displaymath}
\end{lemma}

\begin{proof}
	Let $g$ be the axial isometry of $\gamma_{\tilde{v}}$. Then $\xi$ is a fixed point of $g$.

	Since $\xi\neq\eta$, we can pick some small neighborhood $V_\epsilon\ni\xi$ separated them. Then pick $U_\epsilon\ni\gamma_{\tilde{v}}(\text{--}\infty)$ as in Theorem~\ref{thm_3_2}.

	If $\eta\in U_\epsilon$, Theorem~\ref{thm_3_2} implies this is the required rank $1$ geodesic. If $\eta\notin U_\epsilon$, apply Lemma~\ref{lem_5_3} on $\partial\tilde{M}-V_\epsilon$, we have some $n$ such that $g^{-n}(\partial\tilde{M}-V_\epsilon)\subseteq U_\epsilon$. Then we can find a rank $1$ connecting geodesic $\gamma'$ of $g^{-n}\eta$ and $g^{-n}\xi=\xi$. $g^n\gamma'=\gamma_{\eta,\xi}$ is the required one.
\end{proof}

Now we are going to prove the topological transitivity restricted to $\Omega_1$.

\begin{proof}[Proof of Theorem~\ref{thm_5_0}]
	We will show that for any two relatively open sets $U_1, U_2\subset\Omega_1$, there is an orbit connecting points in $U_1$ and $U_2$, which implies the transitivity.

	Theorem~\ref{thm_4_3} the Anosov Closing Lemma on $\Omega_1$ implies that the periodic orbits are dense in $\Omega_1$. There are two rank $1$ periodic unit vectors $v_1\in U_1$ and $v_2\in U_2$. Assume that $v_1\neq-\phi_{t}(v_2)$ for any $t\in\mathbb{R}$. Lift them to the universal cover as shown in Figure~\ref{fig12}.

	\begin{figure}[htbp]
		\centering
		\scalebox{0.6}{
			\begin{tikzpicture}
				\draw (0,0) circle (4);
				\coordinate [label=above right:$\scriptscriptstyle \gamma_{\tilde{v}_2}(+\infty)$] (A) at (1.936,3.5);
				\node at (A) [fill, circle, inner sep=0.5pt] {};
				\coordinate [label=below right:$\scriptscriptstyle \gamma_{\tilde{v}_2}(\text{--}\infty)$] (B) at (1.936,-3.5);
				\node at (B) [fill, circle, inner sep=0.5pt] {};
				\coordinate [label=above left:$\scriptscriptstyle \gamma_{\tilde{v}_1}(+\infty)$] (C) at (-1.936,3.5);
				\node at (C) [fill, circle, inner sep=0.5pt] {};
				\coordinate [label=below left:$\scriptscriptstyle \gamma_{\tilde{v}_1}(\text{--}\infty)$] (D) at (-1.936,-3.5);
				\node at (D) [fill, circle, inner sep=0.5pt] {};

				\draw (B).. controls (1,-1) and (1,1)..(A);
				\draw (D).. controls (-1,-1) and (-1,1)..(C);

				\coordinate [label=left:$\scriptstyle \tilde{v}_1$] (V) at (-1.23,0);
				\node at (V) [fill, circle, inner sep=0.5pt] {};
				\draw [->] (V) -- (-1.15,0.4);

				\coordinate [label=right:$\scriptstyle dg_2^n\tilde{v}_2$] (U) at (1.58, 2.43);
				\node at (U) [fill, circle, inner sep=0.5pt] {};
				\draw [->] (U) -- (1.7, 2.62);

				\coordinate [label=right:$\scriptstyle \tilde{v}_2$] (W) at (1.58, -2.43);
				\node at (W) [fill, circle, inner sep=0.5pt] {};
				\draw [->] (W) -- (1.45, -2.22);

				\draw [blue] (D).. controls (0, -3) and (1,2).. (A);
				\draw [red] (B).. controls (0,-1) and (-1,3)..(C);

				\coordinate [label=right:$\scriptstyle \tilde{v}_3$] (P) at (-0.7,-2.5);
				\node at (P) [fill, circle, inner sep=0.5pt] {};
				\draw [blue, ->] (P) -- (-0.4,-2.2);

				\coordinate [label=right:$\scriptstyle \tilde{v}_4$] (Q) at (-1.2,2.5);
				\node at (Q) [fill, circle, inner sep=0.5pt] {};
				\draw[red, ->] (Q) -- (-1.3,2.9);

				\coordinate [label=left:$\scriptstyle \tilde{v}_{32}$] (S) at (1.45,2.5);
				\node at (S) [fill, circle, inner sep=0.5pt] {};
				\draw[blue,->] (S) -- (1.5,2.75);

				\coordinate [label=below:$\scriptscriptstyle g_2^{\text{--}n}\gamma_{\tilde{v}_1}(\text{--}\infty)$] (X) at (1.3,-3.78);
				\node at (X) [fill, circle, inner sep=1pt] {};

				\coordinate [label=above:$\scriptscriptstyle g_2^{m}\gamma_{\tilde{v}_1}(+\infty)$] (Y) at (0,4);
				\node at (Y) [fill, circle, inner sep=1pt] {};

				\coordinate [label=above right:$\scriptscriptstyle dg_2^m\tilde{v}_4$] (I) at (1.6,-2.2);
				\node at (I) [fill, circle, inner sep=0.5pt] {};
				\draw [red,->] (1.6,-2.2)--(1.5,-2);

				\coordinate [label=above left:$\scriptscriptstyle dg_2^{\text{--}n}\tilde{v}_3$] (J) at (1.4,-2.1);
				\node at (J) [fill, circle, inner sep=0.5pt] {};
				\draw [blue,->] (J)--(1.3,-1.95);

			\end{tikzpicture}
		}
		\caption{On Universal Cover $\tilde{M}$}\label{fig12}
	\end{figure}
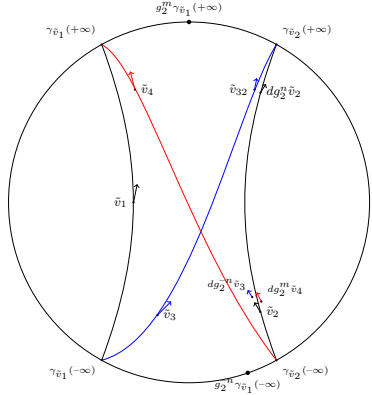

	Using Lemma~\ref{lem_5_4}, we can find a rank $1$ unit vector $\tilde{v}_3$ such that $\gamma_{\tilde{v}_3}$ connecting $\gamma_{\tilde{v}_1}(\text{--}\infty)$ and $\gamma_{\tilde{v}_2}(+\infty)$, and a rank $1$ unit vector $\tilde{v}_4$ such that $\gamma_{\tilde{v}_4}$ connecting $\gamma_{\tilde{v}_2}(\text{--}\infty)$ and $\gamma_{\tilde{v}_1}(+\infty)$.

	Let $v_3$ and $v_4$ denote the projection of $\tilde{v}_3$ and $\tilde{v}_4$ onto $M$. The geodesic $\gamma_{v_3}$ is positively asymptotic to $\gamma_{v_2}$ and negatively asymptotic to $\gamma_{v_1}$, while $\gamma_{v_4}$ behaves the opposite.

	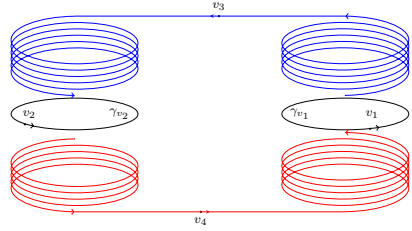
\begin{figure}[htbp]
		\centering
		\scalebox{0.6}{
			\begin{tikzpicture}
				\draw [ decoration={markings, mark=at position 0.6 with {\arrow{>}}},postaction={decorate}] (-3,0) ellipse (40pt and 10pt);
				\node at (-4.1,-0.23) [fill, circle, inner sep=0.5pt] {};
				\node at (-2,0) {$\scriptstyle\gamma_{v_2}$};
				\node at (-4,0) {$\scriptstyle v_2$};
				\draw [ decoration={markings, mark=at position 0.875 with {\arrow{>}}}, postaction={decorate}] (3,0) ellipse (40pt and 10pt);
				\node at (3.55,-0.33) [fill, circle, inner sep=0.5pt] {};
				\node at (2,0) {$\scriptstyle\gamma_{v_1}$};
				\node at (3.6,0) {$\scriptstyle v_1$};
				\coordinate [label=above:$\scriptstyle v_3$] (A) at (0.2,2.17);
				\node at (A) [fill, circle, inner sep=0.5pt] {};
				\draw[blue,decoration={aspect=0.35, segment length=4pt, amplitude=40pt,coil},decorate,arrows = {<[bend]-}] (-3,0.41)--(-3,2.17);
				\draw[blue,decoration={aspect=0.35, segment length=4pt, amplitude=40pt,coil},decorate,arrows={<[bend]-}] (3,2.17)--(3,0.4);
				\draw[blue, midway_mod] (3,2.17)--(-3,2.17);
				\coordinate [label=below:$\scriptstyle v_4$] (B) at (-0.2,-2.17);
				\node at (B) [fill, circle, inner sep=0.5pt] {};
				\draw[red,decoration={aspect=0.35, segment length=4pt, amplitude=40pt,coil},decorate,arrows = {<[bend]-}] (-3,-2.17)--(-3,-0.54);
				\draw[red,decoration={aspect=0.35, segment length=4pt, amplitude=40pt,coil},decorate,arrows={<[bend]-}] (3,-0.41)--(3,-2.17);
				\draw[red, midway_mod] (-3,-2.17)--(3,-2.17);
			\end{tikzpicture}
		}
		\caption{On the Base $M$}\label{fig13}
	\end{figure}

	The last thing left is to show that $v_3\in\Omega_1$. We only need to show that $v_3$ is non-wandering, because $\gamma_{v_3}$ is positively asymptotic to $\gamma_{v_2}$, $\tilde{v}_3$ must share the same stable horocycle and the same strong stable manifolds with $\gamma'_{\tilde{v}_2}(t)$ for some $t$, which coincide. Similarly, the negative asymptote to $\gamma_{v_1}$ guarantees the unstable horocycle coincides with the strong unstable manifolds in the lift $T^1\tilde{M}$. We can conclude that if the rank $1$ vector $v_3$ lies in $\Omega$, it must be contained in $\Omega_1$.

	Because $\gamma_{v_3}$ is positively asymptotic to $\gamma_{v_2}$, we can pick a unit vector $v_{32}$ on $\gamma_{v_3}$ which is close to $v_2$ on $T^1M$. Lift $v_{32}$ to $\tilde{v}_{32}$ on the universal cover as in Figure~\ref{fig12}. It might not be close to $\tilde{v}_2$ but we can find some $n$ such that $\tilde{v}_{32}$ is close to $dg_2^n\tilde{v}_2$. Here $g_2$ is the axial isometry on $\gamma_{\tilde{v}_2}$.

	Similarly, as $\gamma_{v_4}$ is negatively asymptotic to $\gamma_{v_2}$, we have some $v_{42}$ close to $v_4$. And lift it to the universal cover, we have $\tilde{v}_{42}$ is close to $dg_2^{-m}\tilde{v}_2$ for some $m$.

	Therefore $dg_2^{-n}\tilde{v}_3$ are close to $\tilde{v}_2$, and $dg_2^{m}\tilde{v}_4$ are close to $\tilde{v}_2$. Using Theorem~\ref{thm_4_1} the local product structure on the universal cover with $dg_2^{-n}\tilde{v}_3$ and $dg_2^m\tilde{v}_4$, we can find a connecting geodesic of $g_2^{m}\gamma_{\tilde{v}_2}(+\infty)$ and $g_2^{-n}\gamma_{\tilde{v}_1}(\text{--}\infty)$.

	Project this geodesic back on the base $M$, say $\gamma_{+}$. We can see this geodesic starting near $v_3$, travels to the small neighborhood of $v_2$, going back and shadows the geodesic $\gamma_{v_4}$.

	Do the same operation with respect to $v_3, v_4$ and $\gamma_{v_1}$, we can find a similar geodesic, say $\gamma_{-}$, starting near $v_4$, travels to the small neighborhood of $v_1$ and going back shadows the geodesic $\gamma_{v_3}$.

	Since $\gamma_{+}$ and $\gamma_{-}$ travels close to $v_3$, we can use the local product structure near $v_3$ to glue two nearby tangent vectors on $\gamma_{+}$ and $\gamma_{-}$ respectively. This geodesic, say $\gamma_u$ with $u$ close to $v_3$, starts near $v_3$, travels close to $v_1$ and $v_2$, and go back to $v_3$ again. Therefore, $v_3$ is non-wandering as required.

	If $v_1=-\phi_{t}(v_2)$ for some $t\in\mathbb{R}$, that is to say, $v_1$ and $v_2$ generate opposite periodic orbits. The above argument fails. Here we need the condition that the geodesic flow has at least three periodic orbits in $\Omega_1$. We can take a third periodic vector $w$ whose orbit different from the given two. Lift to the universal cover, we can find $\tilde{w}_1$ whose orbit positive asymptotic to $\gamma_{\tilde{v}_1}(+\infty)$ and negative asymptotic to $\gamma_{\tilde{w}}(\text{--}\infty)$, and $\tilde{w}_2$ whose orbit positive asymptotic to $\gamma_{\tilde{v}_2}(+\infty)$ and negative asymptotic to $\gamma_{\tilde{w}}(\text{--}\infty)$. Apply the previous proof to $\tilde{w}_1$ and $\tilde{w}_2$ to get the required connecting orbit. And we are done.
\end{proof}

\begin{remark}
	Apply the Anosov Closing Lemma to the shadowing vector $u$ in the last proof, we actually find a periodic geodesic connecting the open sets $U$ and $V$. This implies that the topological transitivity holds to the restriction of $\Omega^{\text{\text{rec}}}_1$. For details of this discussion, see the proof of Theorem~\ref{M_p_generic}.
\end{remark}

\begin{remark}
	We can easily adapt our argument of transitivity to the restriction of $\Omega_{\text{NF}}$ with the assumption that $\Omega_{\text{NF}}$ is open in $\Omega$ and the geodesic flow has at least three periodic orbits. The proof is quite similar and parallel to the proof given above.
\end{remark}

\section{\bf Generic Invariant Measures}\label{sec6}
In this section, we consider some basic applications of the results in the previous sections. This is inspired by Coud\`ene and Schapara's works on the generic measures for geodesic flows on negatively/ non-positively curved Riemannian manifolds (cf.~\cite{CS1, CS2}).

First, we introduce some notations. In this section, we always assume that $M$ is a Riemannian manifold with no conjugate points, which is not necessarily compact but satisfying bounded asymptote and uniform visibility. In the previous sections~\ref{sec4} and~\ref{sec5}, we have shown that the local product structure, the Anosov Closing Lemma with restriction to $\Omega_1$ and topological transitivity on $\Omega_1$ hold for geodesic flows on $M$.

We use $\mathcal{M}^1(E)$ to denote the set of Borel invariant probability measures for the geodesic flow supported on an invariant subset $E\subset T^1M$, and $\mathcal{M}^1_{erg}(E)\subset\mathcal{M}^1(E)$ to denote the set of ergodic probability measures on $E$. We usually consider the set $E$ to be the non-wandering set $\Omega\subset T^1M$ or its invariant subset  $\Omega_1$ or $\Omega_{\text{NF}}$, which we introduce in Section~\ref{sec2}.  We denote the set of normalized Dirac measures that evenly distributed on closed trajectories in $E$ as $\mathcal{M}^1_{p}(E)$, and its convex hull $\text{CH}(\mathcal{M}^{1}_{p}(E))$.

First, we can establish the following theorem:

\begin{theorem}\label{M_p_generic}
	Let $M$ be a Riemannian manifold as stated in the beginning of this section. Assume the geodesic flow has at least three periodic orbits in $\Omega_1$. Then the set $\mathcal{M}^1_{p}(\Omega_{1})\subset \mathcal{M}^1(\Omega_{1})$ is a dense subset.
\end{theorem}

\begin{proof}
	The proof can be divided into two parts. The first part is to prove that $\mathcal{M}^1_{p}(\Omega_{1})$ is dense in $\textrm{CH}(\mathcal{M}^1_{erg}(\Omega_{1}))$, which follows the idea of the proof of Proposition 3.2 in~\cite{CS1}. And the second part is to prove that $\textrm{CH}(\mathcal{M}^1_{erg}(\Omega_{1}))=\mathcal{M}^1(\Omega_{1})$.

	First of all, we note that by the Anosov Closing Lemma, the set of periodic vectors are dense in $\Omega_{1}$. Using the Birkhoff Ergodic Theorem,  we can show that each ergodic probability measure on $\Omega_{1}$ is a limit of a sequence of normalized Dirac measures evenly distributed on closed trajectories in $\Omega_{1}$ (cf.~\cite{CS1} Lemma 2.2). This implies that $\mathcal{M}^{1}_{p}(\Omega_{1})$ is dense in $\mathcal{M}^1_{erg}(\Omega_{1})$.

	Next, we prove that $\mathcal{M}^1_{p}(\Omega_{1})$ is dense in $\text{CH}(\mathcal{M}^1_{erg}(\Omega_{1}))$. It is sufficient to show that $\mathcal{M}^1_{p}(\Omega_{1})$ is dense in $\text{CH}(\mathcal{M}^{1}_{p}(\Omega_{1}))$, for the latter one is obviously dense in $\text{CH}(\mathcal{M}^1_{erg}(\Omega_{1}))$. We only need to show that the convex combination of finitely many normalized Dirac measures on periodic orbits in $\Omega_1$ can be arbitrarily approximated by one normalized Dirac measure on a periodic orbit. To simplify the argument and illustrate the main idea of the proof, we exhibit this by showing how to find this approximation in the case of convex combination of $3$ normalized Dirac measures. Topological transitivity of the geodesic flow on $\Omega_1$ is needed here.

	Let $v_1, v_2, v_3\in \mbox{Per}(\Omega_{1})$ and $l_1, l_2, l_3>0$ be their periods respectively. Without loss of generality, we can assume that adjacent two of them not stay on the orbit but in the opposite direction of each other. That is to say, $v_{i}\neq\phi_t(v_{i+1})$ for any $t\in\mathbb{R}$. Let $c_1, c_2, c_3>0$ with $c_1+c_2+c_3=1$. We want to find a vector $u\in \mbox{Per}(\Omega_{1})$ such that $\mu_u$ is close enough to $c_1\mu_{v_1}+c_2\mu_{v_2}+c_3\mu_{v_3}$. Here $\mu_w, w\in \mbox{Per}(\Omega_{1})$ denotes the normalized Dirac measure evenly distributed on the orbit of $w$. Since $\mathbb{Q}$ is dense in $\mathbb{R}$, we can assume  $c_1, c_2, c_3$ are all rational numbers, i.e., $c_1=\frac{p_1}{q}, c_2=\frac{p_2}{q}, c_3=\frac{p_3}{q}$ for some $p_1, p_2, p_3, q\in\mathbb{Z}^+$.

	By the transitivity of the geodesic flow  on $\Omega_{1}$, we can find vectors $v_{12}, v_{23}, v_{31}\in\Omega_{1}$ such that $v_{ij}$ is close enough to $v_i$ and $\phi_{t_i}(v_{ij})$ close enough to $v_j$ for some $t_i>0$.

	Given any large integer $N>0$, by the local product structure, we can find a vector $v\in\Omega_{1}$ with the following property: $\gamma_v$ shadows $\gamma_{v_1}$ for $Np_1l_1$ length of time and then goes to a small neighborhood of $v_2$ by shadowing the orbit of $v_{12}$; it shadows $\gamma_{v_2}$ for $Np_2l_2$ length of time and goes to a small neighborhood of $v_3$ by shadowing the orbit of $v_{23}$; it shadows $\gamma_{v_3}$ for $Np_3l_3$ length of time and goes to a small neighborhood of $v_{1}$ by shadowing the orbit of $v_{31}$. The method of this construction of the shadowing orbit is the same as the one in the proof of the transitivity (Theorem~\ref{thm_5_0}). Or people can refer to~\cite{Co, CS1}, in which the procedure to construct such orbit $\gamma_v$ is called gluing $\gamma_{v_1}|_{[0,Np_1l_1]}$, $\gamma_{v_{12}}|_{[0,t_1]}$, $\gamma_{v_2}|_{[0,Np_2l_2]}$, $\gamma_{v_{23}}|_{[0,t_2]}$, $\gamma_{v_3}|_{[0,Np_3l_3]}$ and $\gamma_{v_{31}}|_{[0,t_3]}$ consequently.

	By the Anosov Closing Lemma, we can find a vector $u\in\text{Per}(\Omega_1)$ which is sufficiently close to $v$. It is easy to show that when $N>0$ is large enough, $\mu_u$ will be very close to $c_1\mu_{v_1}+c_2\mu_{v_2}+c_3\mu_{v_3}$. 

	We can do this operation for any $n\geq 2$ periodic orbits. Now we get that $\mathcal{M}^1_{p}(\Omega_{1})$ is dense in $\text{CH}(\mathcal{M}^1_{erg}(\Omega_{1}))$. What is left is to show that $\mathcal{M}^1(\Omega_1)=\text{CH}(\mathcal{M}^1_{erg}(\Omega_{1}))$. The main obstacle is that the set $\Omega_1$ may not be Polish, thus $\mathcal{M}^1(\Omega_1)$ may not be Polish, the ergodic decomposition theorem may not hold.

	However, we can consider the subset $\Omega^{\text{rec}}_1$, the set of rank $1$ recurrent vectors in $\Omega_1$, instead. We know that the set of recurrent vectors in $T^{1}M$ is a $G_{\delta}$ set, and all rank $1$ vectors form an open set. As we have shown in Proposition~\ref{thm_3_3}, all rank $1$ recurrent vectors lie in $\Omega_1$. As an intersection of a $G_{\delta}$ set and an open set, $\Omega^{\text{rec}}_1$ is a $G_{\delta}$ set inside $T^{1}M$, thus a Polish space. By the Poincar\'e Recurrence Theorem, all Borel invariant probability measures on $\Omega_1$ are supported on $\Omega^{\text{rec}}_1$.

	This implies $\mathcal{M}^1(\Omega_1)$ is a Polish space and the ergodic decomposition theorem holds. We have $\mathcal{M}^1(\Omega_1)=\text{CH}(\mathcal{M}^1_{erg}(\Omega_{1}))$. This completes the proof of Theorem~\ref{M_p_generic}.
\end{proof}

\begin{remark}
	As we said at the end of Section~\ref{sec5}, this construction procedure actually gives a periodic connecting geodesic of two different periodic orbits in $\Omega_1$, which implies that the topological transitivity holds to the restriction of $\Omega^{\text{rec}}_1$.
\end{remark}

A straightforward corollary of this theorem is that, in $\Omega_1$, the set of ergodic probability measures is dense in the set of invariant probability measures.

\begin{corollary}
	Suppose $M$ is as stated in Theorem~\ref{M_p_generic}. $\mathcal{M}^1_{erg}(\Omega_{1})$ is dense in $\mathcal{M}^1(\Omega_1)$.
\end{corollary}

By Lemma 4.1 in~\cite{CS1}, we know that the set of invariant probability measures which are fully supported on $\Omega^{\text{rec}}_1$ is a dense $G_{\delta}$ subset in $\mathcal{M}^1(\Omega_1)$. We also know that the geodesic flow admits local product structure, Anosov Closing Lemma and topological transitivity to the restriction of $\Omega^{\text{rec}}_1$. By Theorem 4.2 in~\cite{CS1}, $\mathcal{M}^1_{erg}(\Omega^{\text{rec}}_{1})$ is a $G_{\delta}$ subset in $\mathcal{M}^1(\Omega^{\text{rec}}_1)$, which is exactly $\mathcal{M}^1(\Omega_1)$. Since $\mathcal{M}^1_{p}(\Omega^{\text{rec}}_{1})=\mathcal{M}^1_{p}(\Omega_{1})$ is dense in $\mathcal{M}^1(\Omega_1)$, we get that $\mathcal{M}^1_{erg}(\Omega^{\text{rec}}_{1})$ is a dense $G_{\delta}$ subset in  $\mathcal{M}^1(\Omega_1)$. Take the intersection of the two  dense $G_{\delta}$ subsets in $\mathcal{M}^1(\Omega_1)$, we know that there exists a dense $G_{\delta}$ subset in $\mathcal{M}^1(\Omega_1)$, whose elements are ergodic probability measures fully supported on $\Omega^{\text{rec}}_1$. Therefore, we have the following corollary:
\begin{corollary}
	Suppose $M$ is as stated in Theorem~\ref{M_p_generic}. All ergodic probability measure with full support on $\Omega^{\text{rec}}_1$ form a dense $G_{\delta}$ set in $\mathcal{M}^1(\Omega_1)$. Specifically, we can find an ergodic probability measure whose support contains all rank 1 periodic vectors.
\end{corollary}

In~\cite{CS2}, Coud\`ene and Schapara showed that, for the geodesic flows on rank $1$ non-positively curved manifolds, under the assumption that $\Omega_{\text{NF}}$ is open in $\Omega$, the set $\mathcal{M}^1_{p}(\Omega_{\text{NF}})\subset \mathcal{M}^1(\Omega_{\text{NF}})$ is dense, and  set of ergodic probability measures fully supported on $\Omega_{\text{NF}}$ is a residual subset in $\mathcal{M}^1(\Omega_{\text{NF}})$. Moreover, denote the set of invariant probability measures with zero entropy on $\Omega_{\text{NF}}$ by $\mathcal{M}^1_{0}(\Omega_{\text{NF}})$, they showed that  $\mathcal{M}^1_{0}(\Omega_{\text{NF}})\subset \mathcal{M}^1(\Omega_{\text{NF}})$ is also a residual subset. This means that zero entropy is a generic property for invariant probability measures on $\Omega_{\text{NF}}$. Note that to establish the above generic properties of the invariant measures for geodesic flows, the only things we need are the Anosov Closing Lemma, the local product structure and the topological transitivity of the geodesic flow. Thus, these results can be extended to the manifolds with no conjugate points that satisfy our assumptions as we have shown these three properties holds to the restriction on $\Omega_{\text{NF}}$.

\begin{proposition}\label{omegaNF_generic}
	Suppose $M$ is a Riemannian manifold with no conjugate points which satisfies bounded asymptote and uniform visibility. Assume that $\Omega_{\text{NF}}$ is open in $\Omega$ and the geodesic flow has at least three periodic orbits in $\Omega_{\text{NF}}$. Then
	\begin{enumerate}
		\item The set $\mathcal{M}^1_{p}(\Omega_{\text{NF}})\subset \mathcal{M}^1(\Omega_{\text{NF}})$ is a dense subset.
		\item The set of ergodic probability measures fully supported on $\Omega_{\text{NF}}$ is a residual subset in $\mathcal{M}^1(\Omega_{\text{NF}})$.
		\item The set of invariant probability measures with zero entropy $\mathcal{M}^1_{0}(\Omega_{\text{NF}})\subset \mathcal{M}^1(\Omega_{\text{NF}})$ is a residual  subset.
	\end{enumerate}
\end{proposition}
The proof is similar to the proof given in~\cite{CS2}, we omit it here.

\section*{Acknowledgments}
We want to acknowledge Professor Xiaochun Rong at Rutgers University and Professor Shicheng Xu at Capital Normal University (CNU) for helping us clear some ambiguity in geometry, Professor Dong Chen at Ohio State University for sharing us with his most recent works.

Fei Liu is partially supported by Natural Science Foundation of Shandong Province under Grant No.~ZR2020MA017 and appreciate Professor Qiaoling Wei at Capital Normal University for the hospitality provided when visiting CNU.

Xiaokai Liu appreciate Professor Jana Rodriguez Hertz for her help and constructive communications.

Fang Wang is partially supported by Natural Science Foundation of China (NSFC) under Grant No.~11871045 and key research project of the Academy for Multidisciplinary Studies, Capital Normal University.

\bibliographystyle{amsplain}
\bibliography{manuscript}
\end{document}